\documentclass[10pt,reqno]{amsart}

\usepackage{amsmath, amsthm, amssymb}
\usepackage{amsfonts}
\usepackage[ansinew]{inputenc}
\usepackage[dvips]{epsfig}
\usepackage{graphicx}
\usepackage[english]{babel}

\usepackage{cite}
\usepackage{graphicx}
\usepackage{amscd}
\usepackage{xcolor}
\usepackage{bm}
\usepackage{enumerate}

\usepackage{verbatim}
\usepackage{hyperref}
\usepackage{amstext}
\usepackage{latexsym}
\let\oldsqrt\sqrt
\def\sqrt{\mathpalette\DHLhksqrt}
\def\DHLhksqrt#1#2{%
\setbox0=\hbox{$#1\oldsqrt{#2\,}$}\dimen0=\ht0
\advance\dimen0-0.2\ht0
\setbox2=\hbox{\vrule height\ht0 depth -\dimen0}%
{\box0\lower0.4pt\box2}}

\allowdisplaybreaks

\newcommand{\loglap}{L_{\text{\tiny $\Delta \:$}}\!}
\newcommand{\logrel}{ \text{$(I-\Delta)^{\log}$}}

\newcommand{\R}{\mathbb{R}} 
\newcommand{\N}{\mathbb{N}} 

\newcommand{\F}{\mathcal{F}} 

\newcommand{\dist}{\textnormal{dist}} 
\newcommand{\supp}{\textnormal{supp}} 

\renewcommand{\H}{{\mathcal H}}

\renewcommand{\phi}{\varphi}

\newcommand{\g}{\omega}

\newcommand{\cC}{{\mathcal C}}

\newcommand{\cE}{{\mathcal E}}
\newcommand{\cF}{{\mathcal F}}

\newcommand{\cH}{{\mathcal H}}

\newcommand{\cL}{{\mathcal L}}

\newcommand{\cP}{{\mathcal P}}

\newcommand{\cR}{{\mathcal R}}

\newcommand{\cV}{{\mathcal V}}

\theoremstyle{definition}
\newtheorem{defi}{Definition}[section]
\newtheorem{remark}[defi]{Remark}

\theoremstyle{plain} 
\newtheorem{thm}[defi]{Theorem}
\newtheorem{prop}[defi]{Proposition}
\newtheorem{lemma}[defi]{Lemma}

\theoremstyle{definition}

\numberwithin{equation}{section}

 \title{ The  logarithmic  Schr\"odinger operator and associated Dirichlet  problems }

\author[]
{ Pierre Aime Feulefack}

\address{Goethe-Universit\"{a}t Frankfurt, Institut f\"{u}r Mathematik.
Robert-Mayer-Str. 10, D-60629 Frankfurt, Germany,\ and \  African Institute for Mathematical Sciences in Senegal (AIMS Senegal), 
KM2, Route de Joal, B.P. 14 18. Mbour, S\'en\'egal.}
\email{feulefac@math.uni-frankfurt.de,\ \ pierre.a.feulefack@aims-senegal.org}

\begin{document}
\maketitle
\begin{abstract}
In this note, we study the integrodifferential  operator $(I-\Delta)^{\log}$  corresponding to the logarithmic symbol $\log(1+|\xi|^2)$, which is a singular integral operator given by
\[
 (I-\Delta)^{\log} u(x)=d_{N}\int_{\R^N}\frac{u(x)-u(x+y)}{|y|^{N}}\omega(|y|)\ dy,
\]
where $d_N=\pi^{-\frac{N}{2}}$, $\omega(r)=2^{1-\frac{N}{2}}r^{\frac{N}{2}}K_{\frac{N}{2}}(r)$ and $K_{\nu}$ is  the modified Bessel function of  second kind with index $\nu$.
This operator is the L\'evy generator of the variance gamma process and arises as  derivative $\partial_s\Big|_{s=0}(I-\Delta)^s$ of fractional relativistic Schr\"{o}dinger operators at $s=0$. In order to study  associated  Dirichlet problems  in  bounded domains, we first introduce the functional analytic framework and some properties related to $(I-\Delta)^{\log}$, which allow  to characterize the induced  eigenvalue problem and  Faber-Krahn type inequality.   We also derive a decay estimate in $\R^N$ of the   Poisson problem and   investigate  small order  asymptotics  $s\to 0^+$ of Dirichlet eigenvalues and eigenfunctions of  $(I-\Delta)^s$ in a bounded open Lipschitz set. 
\end{abstract}

{\footnotesize
\begin{center}
\end{center}
}

\section{Introduction and main result}\label{intro}
The  present paper is devoted to the  study of the integrodifferential operator    corresponding to the logarithmic symbol $\log(1+|\cdot|^2)$ and associated Dirichlet problems in domains.  This symbol is known in the probability   literature as  the characteristic exponent of the symmetric  variance  gamma process in $\R^N$ and can be seen as  a  sub-class  of increasing L\'evy process \cite{RefRF}. As particular case  of   geometric stable processes $\log(1+|\cdot|^{2s})$ for  $s\in(0,1)$, it plays an important role in the study of Markov process \cite{RefJB} and finds applications to many different fields such as engineering reliability, credit risk theory in structure models, option pricing in mathematical finance \cite{RefL2} and it is used to study the heavy-tailed financial models \cite{RefHRZ,RefDPE,RefTJAK}. It  was recently used  in wave equation to model  damping mechanism  in $\R^N$ (see \cite{RefRCAR}).

Let us emphasize that the associated operator $(I-\Delta)^{\log}$, which we call the logarithmic Schrödinger operator in the following, has been studied extensively in the literature from a probabilistic and potential theoretic point of view, see e.g. \cite{RefMA,RefMRZ,RefL1,RefHRZ,RefPAi,RefRZ} and the references therein. 
The main purpose of the present paper is to give an account on functional analytic properties of this operator from a PDE point of view. So some of the results we present here are not new but are stated under somewhat different assumptions related to the concept of weak solutions. Moreover, we present proofs not relying on probabilistic techniques but instead on purely analytic methods which are to some extend simpler and more accessible to PDE oriented readers.

Integrodifferential operators of order close to zero are getting  increasing interest in the study of linear and nonlinear  integrodifferential equations, see for e.g. \cite{RefMimi, RefMA ,RefEA,RefHT,HS2121,laptev-weth}  with references therein. In particular, the logarithmic Schrödinger operator $(I-\Delta)^{\log}$ has the same singular local behavior as that of the logarithmic Laplacian $L_{\Delta}$ studied in \cite{RefHT}, while it eliminates the integrability problem of $L_{\Delta}$ at infinity.  We recall that  for  compactly supported Dini continuous functions $\phi : \R^N \to \R$, the  logarithmic Laplacian $L_{\Delta}$ is defined by 
\begin{equation}\label{logarithmic}
\loglap\phi(x) = c_N\lim_{\epsilon\to 0}\int_{\R^N\setminus B_{\epsilon}(x)}\frac{\phi(x)1_{B_1(x)}(y)-\phi(y)}{|x-y|^N}\ dy +\rho_N\phi(x),
\end{equation}
with the constants  $c_{N}:=\frac{\Gamma(N/2)}{\pi^{N/2}}$ and $\rho_N:=2\ln 2+\psi(\frac{N}{2})-\gamma$, see \cite{RefHT} for more details. 
Similarly as in \cite{RefHT}, the starting point of the present paper is the observation
\begin{equation}\label{limit-to-0}
\lim_{s\to 0^+}(I-\Delta)^{s}u= u\quad \text{ for }\ u\in C^{2}(\R^N),
\end{equation}
where for $s\in (0,1)$, the operator $(I-\Delta)^{s}$ stands for the relativistic Schr\"{o}dinger operator which, for sufficiently regular function $u:\R^N\to \R$, is  represented via hypersinglar integral (see \cite[page 548]{RefSAO} and \cite{RefMMV})   
\begin{equation}\label{Integral-Def}
 (I-\Delta)^s u(x)=u(x)+d_{N,s}\lim_{\epsilon\to 0^+}\int_{\R^N\setminus B_{\epsilon}(0)}\frac{u(x+y)-u(x)}{|y|^{N+2s}}\omega_{s}(|y|)\ dy,
\end{equation}
where $d_{N,s}= \frac{\pi^{-\frac{N}{2}}4^s}{\Gamma(-s)}$  is a  normalization constant and the  function $\omega_s$ is given by
\begin{equation}\label{Def_omega}
\begin{split}
\omega_{s}(|y|)&=2^{1-\frac{N+2s}{2}}|y|^{\frac{N+2s}{2}}K_{\frac{N+2s}{2}}(|y|)= \int_{0}^{\infty}t^{-1+\frac{N+2s}{2}}e^{-t-\frac{|y|^2}{4t}}dt.
\end{split}
\end{equation}
In particular, if $u\in C^{2}(\R^N)$, then $(I-\Delta)^su(x)$ is well defined by \eqref{Integral-Def} for every $x\in \R^N$. 
 Here the function $K_\nu$ is  the modified Bessel function of the second kind with index $\nu>0$ and it is given by the  expression
\[
K_{\nu}(r)=\frac{(\pi/2)^{\frac{1}{2}}r^{\nu}e^{-r}}{\Gamma(\frac{2\nu+1}{2})}\int_{0}^{\infty}e^{-rt}t^{\nu-\frac{1}{2}}(1+{t}/{2})^{\nu-\frac{1}{2}}dt.
\]
The normalization constant $d_{N,s}$ in \eqref{Integral-Def} is chosen such that the operator $(I-\Delta)^s$ is equivalently defined via its Fourier representation  given by
\begin{equation}\label{Fourier-logrel}
\cF((I-\Delta)^su)(\xi) = (1+|\xi|^2)^{s}\cF(u)(\xi), \quad\text{ for a.e }\; \xi \in \R^N,
\end{equation}
where $\F$  denotes the usual Fourier transform.
It therefore follows from \eqref{limit-to-0} that one may expect a Taylor expansion with respect to  parameter $s$ of the operator $(I-\Delta)^s$ near zero for $u\in C^{2}(\R^N)$ and $x\in \R^N$ as
\[
(I-\Delta)^su(x) = u(x) +s(I-\Delta)^{\log}u(x) +o(s) \quad \text{ as } \ s\to 0^+,
\]
where, the   logarithmic  Schr\"{o}dinger operator $(I-\Delta)^{\log}$  appears as  the first order  term in the above expansion.  
Indeed, we have the following.
 \begin{thm}\label{Result-1}
 Let $u\in C^{\alpha}(\R^N)$ for some $\alpha>0$ and $1< p\le  \infty$. Then 
 \begin{equation}\label{kernel}
 \begin{split}
 (I-\Delta)^{\log}u(x)&=\frac{d}{ds}\Big|_{s=0} [(I-\Delta)^s u](x)\\
 &=d_{N}\int_{\R^N}\frac{u(x)-u(x+y)}{|y|^{N}}\omega(|y|)\ dy= \int_{\R^N}(u(x)-u(x+y))J(y)\ dy,
 \end{split}
  \end{equation}
  for $x\in \R^N$, where  ~  $d_N:=\displaystyle\pi^{-\frac{N}{2}}=-\lim_{s\to 0^+}\frac{d_{N,s}}{s}$, ~~$J(y)  =d_N\frac{\omega(|y|)}{|y|^{N}},$ ~ and 
  \begin{equation}\label{lim-omega}
\begin{split}
\omega(|y|):= 2^{1-\frac{N}{2}}|y|^{\frac{N}{2}}K_{\frac{N}{2}}(|y|)=\int_{0}^{\infty}t^{-1+\frac{N}{2}}e^{-t-\frac{|y|^2}{4t}}dt.
\end{split}
\end{equation}
  
  Moreover,  
 \begin{itemize}
 \item[(i)] If $u\in L^p(\R^N)$ ~ for ~ $1\le p\le \infty$, ~ then ~$(I-\Delta)^{\log}u\in L^p(\R^N)$ and $$\text{$\displaystyle\frac{(I-\Delta)^s u-u}{s}  \ \to \  (I-\Delta)^{\log}u$ \  in \  $L^p(\R^N)$ \  as \  $s\to 0^+$}$$.
 \item[(ii)]$\displaystyle\cF((I-\Delta)^{\log}u)(\xi) =  \log(1+|\xi|^2)\cF(u)(\xi), \qquad \text{for almost every \;\;}\xi \in \R^N$.
 \end{itemize}
 \end{thm}
 \vspace{0.2cm}
 We note that in the particular case $N=1$, it follows from the definition of  $\omega$ in \eqref{lim-omega} (see also \cite[(2.4)]{RefTM} and \cite[Remark 4.5]{RefHRZ})  that  $\omega(r) = \pi^{N/2}e^{-r}$ and 
\begin{equation}\label{Operetor-N=1}
(I-\Delta)^{\log} u(x)= P.V.\int_{\R}\frac{ u(x)-u(y)}{|x-y|}e^{-|x-y|}\ dy.
\end{equation}
We note here that the operator in  \eqref{Operetor-N=1} appears  in  \cite{RefJAN} and is identified as symmetrized  Gamma process (see also \cite[Example 1]{RefVRL}). We stress however that the symbol of this operator is $\log(1+|\xi|^2)$ and not $\log(1+|\xi|)$ as claimed in \cite[Page 183]{RefJAN}.
The representation of  $J$ in \eqref{kernel}   provides an explicit expression for the kernel of the variance Gamma process in $\R^N$ and gives the following asymptotics expansions
\begin{equation}\label{Asymptoti-o}
	\begin{split}
	J(z)\sim\quad\left\{\begin{aligned}
    \pi^{-\frac{N}{2}}\Gamma(\frac{N}{2})|z|^{-N}\hspace{1.5cm}  & \quad\text{  as }\quad |z|\to 0\\
		  \pi^{-\frac{N-1}{2}}2^{-\frac{N-1}{2}} |z|^{-\frac{N+1}{2}}e^{-|z|} &\quad\text{  as }\quad |z|\to \infty.
	\end{aligned}\right.
	\end{split}
	\end{equation} 
Indeed, these expansions follow directly from (\ref{lim-omega}) and the asymptotics expansions  of the modified Bessel  function $K_{\nu}$ (see  Section \ref{basic}),  (see also  \cite[Theorem 3.4 and  3.6]{RefHRZ} for other proof).
  
  The Green function of the operator $(I-\Delta)^{\log}$   is given (see \cite{RefTM,RefPAi}) by
\begin{equation}\label{Green}
G(x) =  \int_0^{\infty} q_{t}(x)\ dt\quad  \quad x\in \R^N,
\end{equation}
where  for  $t>0$, $q_t :\R^N\to \R $ is the density of the symmetry variance Gamma process i.e.,  for all $t>0$ and $x\in \R^N$,
$$
q_t(x)\ge 0, \qquad \int_{\R^N}q_t(x)\ dx=1  \qquad\text{and }\qquad \cF(q_t)(\xi) = e^{-t\log(1+|\xi|^2)}.
$$
It follows  from \eqref{lim-omega} that   for any $t>0$,
\begin{equation}
\begin{split}
q_{t}(x) 
& =\  \frac{2^{1-N}}{\pi^{N/2}\Gamma(t)}\left(\frac{|x|}{2}\right)^{t-\frac{N}{2}}K_{t-\frac{N}{2}}(|x|),
\end{split}
\end{equation}
and the Green function for  $(I-\Delta)^{\log}$ then writes 
\begin{equation}\label{Gree2}
G(x)=\frac{2^{1-N}}{\pi^{N/2}}\int_{0}^{\infty}\frac{1}{\Gamma(t)}\left(\frac{|x|}{2}\right)^{t-\frac{N}{2}}K_{t-\frac{N}{2}}(|x|)\ dt.
\end{equation}
Using the asymptotics expansions for the modified Bessel function (see \eqref{M-Bessl-Asymtotic} Section \ref{basic}), we have the following proposition.
\begin{prop}\label{Decay}  The function $G$  in (\ref{Gree2}) satisfies the asymptotics properties
\begin{equation}
	\begin{split}
	G(x)\sim\quad\left\{\begin{aligned}
    c_N|x|^{-N}\hspace{2.5cm}& \quad\text{  as }\quad |x|\to 0\\
		  c_N2^{\frac{N-1}{2}}\pi^{1/2}|x|^{-\frac{N+1}{2}}e^{-|x|}&\quad\text{  as }\quad |x|\to \infty.
	\end{aligned}\right.
	\end{split}
	\end{equation}
Moreover, for $f\in L^1(\R^N)$,  the solution   $u= G\ast f$ of the equation $(I-\Delta)^{\log}u =f\ \text{in}\ \R^N$ satisfies
\begin{equation}
	\begin{split}
	u(x)=\quad\left\{\begin{aligned}
    O(|x|^{-N})& \quad\text{  as }\quad |x|\to 0\\
		  O(e^{-|x|})&\quad\text{  as }\quad |x|\to \infty.
	\end{aligned}\right.
	\end{split}
	\end{equation}
\end{prop}

 The next task is  the study in weak sense with the source function $f\in L^2(\Omega)$,  the following related Dirichlet  elliptic   problem  in open bounded set $\Omega\subset\R^N$
 \begin{equation}\label{eq1-eigenvalue}
	\begin{split}
	\quad\left\{\begin{aligned}
		(I-\Delta)^{\log}u &= f && \text{ in $\Omega$}\\
		u &=0             && \text{ on }  \mathbb{R}^N\setminus \Omega.
	\end{aligned}\right.
	\end{split}
	\end{equation}
In  order to settle the corresponding functional analytic framework and energy space related to integro-differential  operator $\logrel$, we introduce the following space  
 \[
 \begin{split}
 H^{\log}(\R^N)&=\Big\{u\in L^2(\R^N): \quad\cE_{\g}(u,u)< \infty\Big\}
 \end{split}
 \]
  where with $J$ as in \eqref{kernel},  the bilinear form considered here is given by 
 \[
 \cE_{\g}(u,v):=\frac{1}{2}\int_{\R^N}\int_{\R^N}(u(x)-u(y))(v(x)-v(y))J(x-y)\ dxdy. 
 \] 
 We shall see in Section \ref{basic} that  $H^{\log}(\R^N)$ is a Hilbert space endowed with the scalar product
 \[
 (u,v)\to \langle u,v\rangle_{H^{\log}(\R^N)}= \langle u,v\rangle_{L^2(\R^N)}+\cE_{\g}(u,u),
 \]
 where $\langle u,v\rangle_{L^2(\R^N)}= \int_{\R^N}u(x)v(x)\ dx$ with corresponding norm
 \[
 \|u\|_{H^{\log}(\R^N)}=\left(\|u\|^2_{L^2(\R^N)}+ \cE_{\g}(u,u)\right)^{\frac{1}{2}}.
 \]
  Let $\Omega\subset \R^N$ be a bounded open set of $\R^N$. Here and the following we identify the space $L^2(\Omega )$ with the space of functions $u\in L^2(\R^N)$ with $u\equiv 0$ on $\R^N\setminus \Omega$. We denote by   $\H_0^{\log}(\Omega)$  the completion of $C^{\infty}_c(\Omega)$ with respect to the norm $\|\cdot\|_{H^{\log}(\R^N)}$. We have, by the Riesz representation theorem that problem  \eqref{eq1-eigenvalue}   admits a unique  weak solution  $u\in \H_0^{\log}(\Omega)$ with
  \[
  \cE_{\g}(u,v) = \int_{\Omega}f(x)v(x)\ dx\quad \text{ for  all}\ v\in \H_0^{\log}(\Omega).
  \]
Moreover, if   $f\in L^{\infty}(\Omega)$ and $\Omega$ satisfies a uniform exterior sphere condition,  it follows from the Green function representation   and the regularity estimates in \cite{RefMA,RefMimi,RefPAi}   that   $u\in C_0(\Omega):=\{u\in C(\R^N)\ :\ u=0\text{ on }\ \R^N\setminus \Omega\}$. 

%

We  aim next to study  the  eigenvalue problem in bounded domain $\Omega\subset \R^N$ involving the logarithmic Schr\"{o}dinger operator $(I-\Delta)^{\log}$, that is, we consider (\ref{eq1-eigenvalue}) with $f=\lambda u$.
To avoid a priori regularity assumption, we consider the eigenvalue  problem \eqref{eq1-eigenvalue} in weak sense.
We call a function $u\in \cH^{\log}_0(\Omega)$ an eigenfunction of \eqref{eq1-eigenvalue} corresponding to the eigenvalue  $\lambda$ if 
\begin{equation}\label{weak}
\cE_{\g}(u,\phi)= \lambda\int_{\Omega}u\phi\ dx\qquad \text{ for all } \ \phi\in \cC^{\infty}_c(\Omega).
\end{equation}
We then have the following characterisation of the eigenvalues and eigenfunctions for the operator  $(I-\Delta)^{\log}$ in an open bounded set $\Omega$ of $\R^N$.
\begin{thm}\label{eigenvalue}
Let $\Omega\subset \R^N$ be an open bounded set. Then 
\begin{itemize}
\item[(i)] Problem \eqref{eq1-eigenvalue} admits  an eigenvalue $\lambda_1(\Omega)>0$   characterized  by 
\begin{equation}\label{Lambda-1}
\lambda_{1}(\Omega)=\inf_{\substack{u\in \H_0^{\log}(\Omega)\\u\neq 0}}\frac{\cE_{\g}(u,u)}{\|u\|^2_{L^2(\Omega)}}= \inf_{u\in \cP_1(\Omega)}\cE_{\g}(u,u),
\end{equation}
with $\cP_1(\Omega):=\{u\in \H_0^{\log}(\Omega): \|u\|_{L^2(\Omega)}=1\}$ and there exists a positive function $\phi_1\in \H_0^{\log}(\Omega)$, which is an eigenfunction corresponding to $\lambda_1(\Omega)$ and  that attains the minimum in \eqref{Lambda-1}, i.e. $\|\phi_1\|_{L^2(\Omega)}=1$ and~
$
\lambda_{1}(\Omega)= \cE_{\g}(\phi_1,\phi_1).
$
\item[(ii)] The first eigenvalue $\lambda_1(\Omega)$ is simple, that is, if $u\in  \H_0^{\log}(\Omega)$ satisfies \eqref{weak}  with $\lambda=\lambda_1(\Omega)$,  then $u=\alpha\phi_1$ for some $\alpha\in \R$.
\item[(iii)] Problem \eqref{eq1-eigenvalue} admits a sequence of eigenvalues $\{\lambda_{k}(\Omega)\}_{k\in \N}$ with 
\[
0<\lambda_{1}(\Omega)<\lambda_2(\Omega)\le \cdots\le \lambda_k(\Omega)\le \lambda_{k+1}(\Omega)\cdots,
\]
with corresponding eigenfunctions $\phi_k$, $k\in \N$ and~ 
$
\lim_{k\to \infty}\lambda_{k}(\Omega) = +\infty.
$\\
Moreover, for any $k\in\N$, the  eigenvalue $\lambda_{k}(\Omega)$ can be characterized as 
\begin{equation}\label{Lambda-k}
\lambda_k(\Omega)= \inf_{u\in \cP_k(\Omega)}\cE_{\g}(u,u)
\end{equation}
where $\cP_k(\Omega)$ is given by
\[
\cP_k(\Omega):= \{ u\in \H_0^{\log}(\Omega): \int_{\Omega}u\phi_j\ dx  =0 \ \text{ for \ } j=1,2,\cdots k-1 \ \text{ and } \ \|\phi_k\|_{L^2(\Omega)}=1\}.
\]
\item[(iv)] The sequence $\{\phi_k\}_{k\in\N}$ of eigenfunctions corresponding to eigenvalues $\lambda_k(\Omega)$ form a complete orthonormal basis of $L^2(\Omega)$ and an orthogonal system of $\H_0^{\log}(\Omega)$. 
\end{itemize} 
\end{thm}	
Using  the  $\delta$-decomposition technique introduced in \cite{RefPST}, we provide a boundedness  result of the eigenfunctions introduced in Theorem \ref{eigenvalue}.
\begin{prop}\label{Bound-eigenfunction}
Let $u\in\cH_0^{\log}(\Omega)$ and $\lambda>0$ satisfying \eqref{weak}. 
Then $u\in L^{\infty}(\Omega)$ and there exists a constant $C:=C(N,\Omega)>0$ such that
\[
\|u\|_{L^{\infty}(\Omega)}\le C\|u\|_{L^2(\Omega)}.
\]
\end{prop}
Our next result concerns the Faber-Krahn inequality for the logarithmic Schr\"odinger operator, which says: Among all open sets in $\R^N$ with given measure, ball uniquely gives the smallest first Dirichlet eigenvalue of  the logarithmic Schr\"odinger operator $(I-\Delta)^{\log}$. Here and in the following, we denote by $B^{\ast}$ the open ball in $\R^N$ centered at zero with radius determined such that  $|\Omega|=|B^{\ast}|$ 
\begin{thm}[Faber-Krahn inequality]\label{Faber-Krahn}
Let $\Omega\subset \R^N$ be open and bounded, and  $\lambda_{1,\log}(\Omega)$  be the principal eigenvalue of  $(I-\Delta)^{\log}$ in $\Omega$. Then
\begin{equation}\label{Faber}
\lambda_{1,\log}(\Omega)\ge \lambda_{1,\log}(B^{\ast}).
\end{equation}
Moreover, if equality occurs,  $\Omega$ is a ball.  Consequently, if $\Omega$ is a ball in $\R^N$, the first eigenfunction $\phi_{1,\log}$ corresponding  to $\lambda_{1,\log}(B)$ is radially symmetric.
\end{thm}

Our last  result  concerns   small order asymptotics   $s\to 0^+$  of  eigenvalues and corresponding eigenfunctions  of  the relativistic Schr\"odinger  operator $(I-\Delta)^s$    on  bounded Lipschitz domain  $\Omega\subset\R^N$, which is an analogue, but    a part of the result  of the small order asymptotics $s\to 0^+$ proved in \cite{RefPST}  for the fractional Laplacian.
\begin{thm}\label{convergent-1}
Let $\Omega$ be a bounded Lipschitz domain in $\R^N$, and $\lambda_{k,s}(\Omega)$ resp. $\lambda_{k,\log}(\Omega)$  be  the $k$-th Dirichlet eigenvalue  of $(I-\Delta)^s$  resp. of $(I-\Delta)^{\log}$ on $\Omega$. Then for $s\in (0,1)$, the eigenvalue $\lambda_{k,s}(\Omega)$ satisfies the expansion
\begin{equation}
\label{asymptotic-exp}
\lambda_{k,s}(\Omega)=1 + s\lambda_{k,\log}(\Omega) + o(s) \quad \text{as }\quad s\to 0^+.
\end{equation}
Moreover, if  $(s_n)_n\subset (0,s_0)$, $s_0>0$ is a sequence with $s_n\to 0$ as $n\to \infty$, then if $\psi_{1,s}$ is the unique nonnegative $L^2$-normalized  eigenfunction of $(I-\Delta)^s$ corresponding to  the principal eigenvalue $\lambda_{1,s}(\Omega)$, we have that 
\begin{equation}
\label{convergent-eigen-k=1}
\psi_s\to \psi_{1,\log} \quad \text{in }\ L^2(\Omega)\quad \text{as }\quad s\to 0^+,
\end{equation}
and after passing to a subsequence, we have that 
\begin{equation}
\label{convergent-eigen-k}
\psi_{k,s}\to \psi_{k,\log} \quad \text{in }\ L^2(\Omega)\quad \text{as }\quad s\to 0^+,
\end{equation}
where $\psi_{1,\log}$,~  resp.  $\psi_{k,\log}$, $k\ge 2$ is  the unique nonnegative $L^2$-normalized  eigenfunction resp. a $L^2$-normalized  eigenfunction  corresponding to $\lambda_{1,\log}(\Omega)$ resp.  to $\lambda_{k,\log}(\Omega)$.
\end{thm}

\vspace{0.2cm}

The paper is organized as follows. In Section \ref{basic}, we provide the proof of Theorem \ref{Result-1} and establish some properties of  $(I-\Delta)^{\log}$ and functional spaces.  In Section \ref{Sec-eigenvalue-problem}, we prove Theorem \ref{eigenvalue}  and, using the $\delta$-decomposition tecnique  introduced in \cite{RefPST}, we give the proof of  Proposition \ref{Bound-eigenfunction} on the $L^{\infty}$-bound of eigenfunctions and close the section with the  proof of   Theorem \ref{Faber-Krahn} on  Faber-Krahn inequality. Section 
\ref{Sect-small} is dedicated to the proof of  Theorem \ref{convergent-1} on small order asymptotics $s\to 0^+$ of   the eigenvalues and corresponding eigenfunctions  of  $(I-\Delta)^s$.   In section \ref{Sect-Regularity}, we establish the proof of   Proposition \ref{Decay} concerning the  decay  of the solution of  Poisson problem in $\R^N$.  Finally,  Section \ref{Additional}  collects some theorems that can be directly deduced from known results in the literature. \\

\textbf{Notation}: ~We let $\omega_{N-1}= \frac{2\pi^{\frac{N}{2}}}{\Gamma(\frac{N}{2})} $ denote the measure of the unit sphere in $\R^N$ and,  for a set $A \subset\R^N$ and $x \in \R^N$, we define $\delta_A(x):=\dist(x,A^c)$ with $A^c=\R^N\setminus A$ and, if $A$ is measurable, then $|A|$ denotes its Lebesgue measure. Moreover, for given $r>0$, let $B_r(A):=\{x\in \R^N\;:\; \dist(x,A)<r\}$, and let $B_r(x):=B_r(\{x\})$ denote the ball of radius $r$ with $x$ as its center. If $x=0$ we also write $B_r$ instead of $B_r(0)$.
If $A$ is open, we denote by $C^k_c(A)$ the space of function $u:\R^N\to \R$ which are $k$-times continuously differentiable and with support compactly contained in $A$. If $f$ and $g$ are two functions,
then, $f\sim g$ as $x\to a$ if  $\frac{f(x)}{g(x)}$ converges to a constant as $x$ converges to $a$.\\

\textbf{Acknowledgements.}~ 
This work is supported by DAAD and BMBF (Germany) within the project 57385104. The author would like to thank     Mouhamed Moustapha Fall,  Sven Jarohs and Tobias Weth  for helpful 
discussions and comments.

\section{ Properties of the  operator and Functional spaces }\label{basic}
We commence this section with the establishment of the   integral representation of the operator $(I-\Delta)^{\log}$ for a function $u\in C^{\alpha}(\R^N)$,  that is, we provide the proof of Theorem \ref{Result-1}.  After that, we also provide some properties of the functional spaces related to $(I-\Delta)^{\log}$.  We first introduce the following asymptotics approximations 
 (see \cite{RefFDRC}) for the modified Bessel function  $K_{\nu}$. It well-known that 
\begin{equation}\label{M-Bessl-Asymtotic}
\begin{split}
K_{\nu}(r)&\sim
\begin{cases}
2^{|\nu|-1} \Gamma(|\nu|)r^{-|\nu|}, \quad r\to 0,\quad \nu\neq 0,\\
\log\frac{1}{r}, \qquad\qquad\qquad r\to 0,\quad \nu = 0,\\
 \sqrt{{\pi}/{2}}r^{-\frac{1}{2}}e^{-r}, \quad r\to +\infty,
 \end{cases}
\end{split}
\end{equation}
and the monotonicity (see \cite[10.37.1]{RefFDRC})
\begin{equation}\label{Monotonicity-Bessel}
|K_{\nu}(r)|< |K_{\mu}(r)|\quad\text{ for }\quad 0\le \nu< \mu.
\end{equation}
Consequently,
\begin{equation}\label{omega-Asymtotic}
\begin{split}
\omega_{s}(r)&\sim
\begin{cases}
 \Gamma(\frac{N+2s}{2}), \quad r\to 0,\quad \\
 2^{-\frac{N+2s-1}{2}}r^{\frac{N+2s-1}{2}}e^{-r}, \quad r\to +\infty.
 \end{cases}
\end{split}
\end{equation}
Note also that the functions $s\mapsto \omega_s$ and $s\mapsto d_{N,s}$ defined in \eqref{Integral-Def} are continuous  function of $s$ and we have that  $\displaystyle\lim_{s\to 0^+} d_{N,s}= 0$ and, by dominated convergent theorem,
\begin{equation}\label{lim-omega-p}
\begin{split}
\omega(|y|):=\lim_{s\to 0^+}\omega_{s}(|y|)&= 2^{1-\frac{N}{2}}|y|^{\frac{N}{2}}K_{\frac{N}{2}}(|y|)=\int_{0}^{\infty}t^{-1+\frac{N}{2}}e^{-t-\frac{|y|^2}{4t}}dt.
\end{split}
\end{equation}
 We now give the 
\begin{proof}[{Proof of Theorem} \ref{Result-1}]
  Let $u\in C^{\alpha}(\R^N)$ with   $0<s<\min\{\frac{\alpha}{2},\frac{1}{2}\}$. Then, from the definition of $(I-\Delta)^s$ in \eqref{Integral-Def}, the principal value can be dropped out and  we have the different quotient
 \begin{align*}
 \frac{(I-\Delta)^s u-u}{s}&=\frac{d_{N,s}}{s}\int_{\R^N}\frac{u(x+y)-u(x)}{|y|^{N+2s}}\omega_{s}(|y|)\ dy = A_{\epsilon}(s,x)+D_{\epsilon}(s,x),
 \end{align*}
 where $\epsilon>0$,   with $A_{\epsilon}(s,x)$ and $D_{\epsilon}(s,x)$ given respectively by
 \[
A_{\epsilon}(s,x):= \frac{d_{N,s}}{s}\int_{|y|<\epsilon}\frac{u(x+y)-u(x)}{|y|^{N+2s}}\omega_{s}(|y|)\ dy,
 \]
 \[
 D_{\epsilon}(s,x):= \frac{d_{N,s}}{s}\int_{|y|\ge \epsilon}\frac{u(x+y)-u(y)}{|y|^{N+2s}}\omega_{s}(|y|)\ dy.
 \]
 First, from  \eqref{Def_omega} and \eqref{lim-omega}   and  the fact that  $|y|^{-2s}\le \epsilon^{-2}$ for $|y|\ge \epsilon$ and $s\in (0,1)$, we have by dominated convergent theorem that
\[
 D_{\epsilon}(s,x)=\frac{d_{N,s}}{s}\int_{|y|\ge \epsilon}\frac{u(x+y)-u(x)}{|y|^{N+2s}}\omega_{s}(|y|)\ dy \to {D}_{\epsilon}(0,x)\quad \text{ as }\ s\to 0^+,
\]
with 
\[
\begin{split}
{D}_{\epsilon}(0,x)&:=d_N\int_{|x-y|\ge \epsilon}\frac{u(x)-u(y)}{|x-y|^{N}}\omega(|x-y|)\ dy=\int_{|x-y|\ge \epsilon}(u(x)-u(y))J(x-y)\ dy.
\end{split}
\]
Since next $u\in C^{\alpha}(\R^N)$, it also follows that 
\begin{align*}
A_{\epsilon}(s,x)=\frac{d_{N,s}}{s}\int_{|y|<\epsilon}\frac{u(x+y)-u(x)}{|y|^{N+2s}}\omega_{s}(|y|)\ dy\to A_{\epsilon}(0,x)\quad\text{ as} \quad s\to 0^+, 
\end{align*}
with
\begin{align*}
A_{\epsilon}(0,x)=d_N\int_{|y|<\epsilon}\frac{u(x)-u(x+y)}{|y|^{N}}\omega(|y|)\ dy=\int_{|x-y|< \epsilon}(u(x)-u(y))J(x-y)\ dy.
\end{align*}
We recall  that $\lim_{s\to 0}d_{N,s}/s= -d_N$. It is easy to see that $A_{\epsilon}(0,x)\to 0$ as $\epsilon\to 0^+$, and   from the the fact that  $u\in C^{\alpha}(\R^N)$, we also infer that 
\begin{align*}
\left|(I-\Delta)^{\log}u(x)-D_{\epsilon}(0,x)  \right|\le C\int_{|y|< \epsilon}\min\{1,|y|^{\alpha}\}\ dy\to 0\quad \text{ as}\ \epsilon\to 0^+.
\end{align*}
Since $u\in C^{\alpha}(\R^N)$, setting $\kappa_{N,s,u}=\Big|\frac{d_{N,s}}{s}\Big| \Gamma(({N+2s})/{2})\|u\|_{C^{\alpha}(\R^N)}\omega_{N-1}$ it follows from \eqref{omega-Asymtotic} that
\begin{align*}
|A_{\epsilon}(s,x)|&\le \Big|\frac{d_{N,s}}{s}\Big|\int_{|y|<\epsilon}\frac{\|u\|_{C^{\alpha}(\R^N)}}{|y|^{N+2s-\alpha}}\omega_{s}(|y|)\ dy\le\kappa_{N,s,u}\frac{\epsilon^{\alpha-2s}}{\alpha-2s}.
\end{align*}
Consequently,
\[
\|A_{\epsilon}(s,\cdot)\|_{L^p(B_{\epsilon})}\le \kappa_{N,s,u}\frac{\epsilon^{\frac{N}{p}+\alpha-2s}}{\alpha-2s} \quad \text{for $1\le p\le \infty$}.
\]
On the other hand,  using again  \eqref{omega-Asymtotic} with $s=0$, we infer that
\begin{align*}
|{D}_{\epsilon}(0,x)|&\le \int_{|x-y|\ge \epsilon}|u(x)-u(x+y)|J(y)\ dy\\
&\le 2\|u\|_{C^{\alpha}(\R^N)}\big(\int_{B_1\setminus B_{\epsilon}}|y|^{\alpha-N}dy+\int_{|y|\ge 1} e^{-|y|}\ dy\big)\\
&\le2\|u\|_{C^{\alpha}(\R^N)}\Big(2\frac{1-\epsilon^{\alpha}}{\alpha}+\omega_{N-1}\Gamma(N,1)\Big)
= \ C_{N,\epsilon} \|u\|_{C^{\alpha}(\R^N)}.
\end{align*}
Therefore,
\[
\|{D}_{\epsilon}(0,\cdot)\|_{L^{\infty}(\R^N\setminus B_{\epsilon})}<\infty.
\]
Next, by the Minkowski's  integral inequality,  we have 
\begin{align*}
&\|{D}_{\epsilon}(0,\cdot)\|_{L^p(\R^N\setminus B_{\epsilon})}\le \Big(\int_{\R^N\setminus B_{\epsilon}} \Big | \int_{|y|\ge \epsilon}(u(x)-u(x+y))J(y)\ dy \Big|^pdx \Big)^{\frac{1}{p}}\\
&\hspace{1cm}\le \int_{\R^N\setminus B_{\epsilon}} \Big( \int_{\R^N\setminus B_{\epsilon}}|u(x)-u(x+y)|^p\ dx \Big)^{\frac{1}{p}}J(y) \ dy\\
&\hspace{1cm}\le 2^{\frac{p-1}{p}}\|u\|_{L^p(\R^N\setminus B_{\epsilon})}\int_{\R^N\setminus B_{\epsilon}}J(y)\ dy<\infty.
\end{align*}
Therefore, we conclude that ${D}_{\epsilon}(0,\cdot)\in L^p(\R^N\setminus B_{\epsilon})$ $\text{ for  all}\quad 1\le  p\le  \infty$~ and thus
\begin{equation}\label{Uniform}
\|D_{\epsilon}(s,\cdot)-{D}_{\epsilon}(0,\cdot)\|_{L^p(\R^N\setminus B_{\epsilon})}\to 0 \quad \text{uniformly in $\epsilon$ \  as } s\to 0^+.
\end{equation}
This allows  to conclude  for $x\in \R^N$ that
\begin{equation}\label{Equality}
\lim_{\epsilon\to 0^+}D_{\epsilon}(0,x) =\lim_{\epsilon\to 0^+}\int_{|y|\ge \epsilon}(u(x)-u(x+y))J(y)\ dy=(I-\Delta)^{\log}u(x).
\end{equation}
Taking into account  the above facts, we find with $1\le p< \infty$ that
\begin{align*}
&\left\|\frac{(I-\Delta)^s u-u}{s}-(I-\Delta)^{\log}u\right\|_{L^p(\R^N)}=\left\|A_{\epsilon}(s,\cdot)+D_{\epsilon}(s,\cdot)-(I-\Delta)^{\log}u\right\|_{L^p(\R^N)}\\
&\hspace{3cm}\le\left\|A_{\epsilon}(s,\cdot)\right\|_{L^p(\R^N)}~+~\left\|D_{\epsilon}(s,\cdot)-(I-\Delta)^{\log}u\right\|_{L^p(\R^N)}\\
&\hspace{3cm}\le \kappa_{N,s,u}\frac{\epsilon^{\frac{N}{p}~+~\alpha-2s}}{\alpha-2s}+\left\|D_{\epsilon}(s,\cdot)-(I-\Delta)^{\log}u\right\|_{L^p(\R^N)}.
\end{align*}
Therefore,  using \eqref{Uniform} and \eqref{Equality} , we have for every $1\le p< \infty$ that
\[
\limsup_{s\to 0^+}\left\|\frac{(I-\Delta)^s u-u}{s}-(I-\Delta)^{\log}u\right\|_{L^p(\R^N)}\le \kappa_{N,u}\frac{\epsilon^{\frac{N}{p}+\alpha}}{\alpha}\quad\text{for every $\epsilon>0$},
\]
where  $\kappa_{N,u}$ is independent of $\epsilon$.~  The case $p=\infty$ follows by the same computation and
\[
\limsup_{s\to 0^+}\left\|\frac{(I-\Delta)^s u-u}{s}-(I-\Delta)^{\log}u\right\|_{L^{\infty}(\R^N)}\le \kappa_{N,u}\frac{\epsilon^{\alpha}}{\alpha}\quad\text{for every $\epsilon>0$}.
\]
 Moreover, it follows from the arbitrary of  $\epsilon $ that
\[
\lim_{s\to 0^+}\left\|\frac{(I-\Delta)^s u-u}{s}-(I-\Delta)^{\log}u\right\|_{L^p(\R^N)}=0 \quad\text{for every $1\le p\le \infty$}.
\]
This completes the of item $(i)$. The proof of 
item $(ii)$ is a particular case with $p=2$. Moreover,  using the continuity of the Fourier transform   in $L^2(\R^N)$, we have that
 \begin{align*}
 \cF((I-\Delta)^{\log}u) = \lim_{s\to 0^+}\frac{\cF((I-\Delta)^su)-\cF(u)}{s}&=\lim_{s\to 0^+}\left(\frac{(1+|\cdot|^2)^{s}-1}{s}\right)\cF(u)\\
 &=\log(1+|\cdot|^2)\cF(u)\quad\text{ in }\quad L^2(\R^N).
  \end{align*}
 We therefore infer that
 \[
  \cF((I-\Delta)^{\log}u)(\xi)= \log(1+|\cdot|^2)\cF(u)(\xi), \qquad \text{for almost every \;\;}\xi \in \R^N.
 \]
 The proof of Theorem \ref{Result-1} is henceforth completed.
 \end{proof}
 In the following,  we  let $L_0(\R^N)$ denotes the space 
\[
L_{0}(\R^N):= \Big\{ u: \ \R^N\to \R :\   \|u\|_{L_0(\R^N)} <\infty\Big\} \ \text{ with }\ \|u\|_{L_0(\R^N)}= \int_{\R^N}\frac{|u(x)|e^{-|x|}}{(1+|x|)^{\frac{N+1}{2}}}\ dx.
\] 
Let $U$ be a measurable subset and $u:U\to \R$ be a measurable function. The modulus of continuity of $u$ at a point $x\in U$ is defined by 
\[
\omega_{u,x,U}: (0, +\infty)\to [0,+\infty), \qquad \omega_{u,x,U}(r)=\sup_{\substack{y\in U,\ |x-y|\le r}}|u(x)-u(y)|.
\]
The function $u$ is called Dini continuous at $x$ if 
$$\int_{0}^1\frac{\omega_{u,x,U}(r)}{r}\ dr<\infty.$$ 
 Moreover, we call $u$ uniformly Dini continuous in $U$ for the uniform modulus of continuity 
\[
\omega_{u,U}(r):= \sup_{x\in U}\omega_{u,x,U}(r)\quad\text{ if }\qquad \int_{0}^1\frac{\omega_{u,U}(r)}{r}\ dr<\infty.
\]
 In the following proposition, we  list some  properties  the operator $(I-\Delta)^{\log}$.  
\begin{prop}\label{well-defined}
\begin{itemize}
\item[(i)]Let $u\in L_0(\R^N)\cap L^{\infty}(\R^N)$. If $u$ is locally  Dini continuous at some point $x\in \R^N$, then the  operator $(I-\Delta)^{\log}u$~ is well defined by  
\[
(I-\Delta)^{\log}u(x)=\int_{\R^N}(u(x)- u(y))J(x-y)\ dy.
\]
\item[(ii)] Let $\phi\in C_c^{\alpha}(\R^N)$ for some $\alpha>0$,  there is $C=C(N,\phi)$ such that 
\[
|(I-\Delta)^{\log}\phi(x)|\le C\|\phi\|_{C^{\alpha}(\R^N)}\frac{e^{-|x|}}{(1+|x|)^{\frac{N+1}{2}}}.
\] 
In particular, for $u\in L_0(\R^N)$,~ $\logrel u$~ defines a distribution via the map 
\[
\phi\mapsto\langle \logrel u,\phi\rangle=\int_{\R^N} u(I-\Delta)^{\log}\phi\ dx.
\]
\item[(iii)]Let  $u\in L_0(\R^N)$ and  $r>0$ such that $u\in C^{\alpha}(B_r(0))$ for some $\alpha>0$. Then there exists a constant $C:=C(N,\alpha)>0$ such that
\[
|(I-\Delta)^{\log}u(x)|\le C(\|u\|_{C^{\alpha}(B_{r(0)})}+\|u\|_{\cL_{0}(\R^N)})
\]
\item[(iv)] If $u\in C^{\beta}(\R^N)$ for some $\beta>0$, then $(I-\Delta)^{\log}u\in  C^{\beta-\epsilon}(\R^N)$ for every $\epsilon$ such that $0<\epsilon<\beta$  and there exists a constant $C:=C(N,\beta,\epsilon)>0$ such that
\[
[(I-\Delta)^{\log}u]_{\beta-\epsilon}\le C\|u\|_{C^{\beta}(\R^N)}.
\]
\item[(v)] Let $\phi,\psi\in\cC_c^{\infty}(\Omega)$. Then we have the product rule  
\begin{align*}
 (I-\Delta)^{\log}(\phi\psi)(x)&=\phi(x)(I-\Delta)^{\log}\psi(x) + \psi(x)(I-\Delta)^{\log}\phi(x)-\Lambda(\phi,\psi).
\end{align*}
with 
$$
\Lambda(\phi,\psi):=\int_{\R^N}(\phi(x)-\phi(y))(\psi(x)-\psi(y))J(x-y)\ dy
$$
If $\rho_{\epsilon}$, $\epsilon>0$ is a  family of  mollified, then
\[
[(I-\Delta)^{\log}(\rho_{\epsilon}\ast \phi)](x) = \rho_{\epsilon}\ast[ (I-\Delta)^{\log}\phi](x).
\] 
\end{itemize}
\end{prop}
\begin{proof}
Let $x\in \R^N$. By splitting  the integral and using the asymptotic of $J$   in \eqref{Asymptoti-o}, we have the following,
\begin{align*}
&|(I-\Delta)^{\log}u(x)|\le \int_{B_1(x)}{|u(x)- u(y)|}J(x-y)\ dy +\int_{\R^N\setminus B_1(x)}(|u(x)|+|u(y)|)J(x-y)\ dy\\
&\le \Gamma(N/2)\omega_{N-1}\int_{0}^{1}\frac{\omega_{u,x}(r)}{r}\ dr
+C\|u\|_{L^{\infty}}\int_{\R^N\setminus B_1}e^{-|y|}\ dy+C\int_{\R^N\setminus B_1(x)}\frac{|u(y)|e^{-|x-y|}}{|x-y|^{\frac{N+1}{2}}}\ dy\\
&\le C(1+\|u\|_{L^{\infty}(\R^N)})+C\Big(\int_{B_{1+2|x|}(0)\setminus B_1(x)}+\int_{\R^N\setminus B_{1+2|x|}(0)}\Big)\frac{|u(y)|e^{-|x-y|}}{|x-y|^{\frac{N+1}{2}}}\ dy\\
&\le C(1+\|u\|_{L^{\infty}(\R^N)})+C\int_{B_{1+2|x|}(0)}|u(y)|\ dy+C\int_{\R^N\setminus B_{1+2|x|}(0)}\frac{|u(y)|e^{-|x-y|}}{|x-y|^{\frac{N+1}{2}}}\ dy.
\end{align*}
Now, since  $|x-y|\ge \frac{1}{2}(1+|y|)$ for $|y|\ge 1+2|x|$, it follows that
\begin{align*}
|(I-\Delta)^{\log}u(x)|\le  C(1+\|u\|_{L^{\infty}(\R^N)}+ \|u\|_{L_0(\R^N)})<\infty.
\end{align*}
This shows that $(I-\Delta)^{\log}u(x)$ is well-defined.\\
To prove $(ii)$, for $x\in \R^N$, we use again \ref{Asymptoti-o} and  the representation 
\[
(I-\Delta)^{\log}\phi(x)= \frac{ d_N}{2}\int_{\R^N}\frac{2\phi(x)-\phi(x+y)- \phi(x-y)}{|y|^N}\omega(|y|)\ dy.
\]
Put $A:= \|\phi\|_{C^{\alpha}(\R^N)}$. Note first that, since $\phi\in C_c^{\alpha}(\R^N)$, we have 
\[
|2\phi(x)-\phi(x+y)-\phi(x-y)|\le A\min\{1,|y|^{\alpha}\}.
\]
Therefore, for any $x\in \R^N$, we have with $0<r<1$ that  
\begin{align*}
|(I-\Delta)^{\log}\phi(x)|&\le\frac{ d_N}{2}\int_{\R^N}\frac{|2\phi(x)-\phi(x+y)- \phi(x-y)|}{|y|^N}\omega(|y|)\ dy\\
&\le A\int_{\R^N}\frac{\min\{1,|y|^{\alpha}\}}{|y|^N}\omega(|y|)\ dy \\
&\le  C_NA\Big(\int_{B_r}|y|^{\alpha-N}\ dy+\int_{B_1\setminus B_r}\frac{1}{|y|^N}\ dy+\int_{\R^N\setminus B_1}e^{-|y|}\ dr \Big)\\
&\le C(N,r,\alpha)A.
\end{align*} 
Next,  Let $R>0$ be such that   $B_{1}(\supp\ \phi)\subset B_{R}(0)$.  Let $x\in \R^N$ satisfying  $\frac{|x|}{2}>R$, then $1+|y|\le \frac{|x|}{2}$ for $y\in B_{1}(\supp\ \phi)$ and  $|x-y|\ge |x|-|y|\ge \frac{|x|}{2}+1\ge \frac{1}{2}(|x|+1)$.
Moreover, since $\phi(x)\equiv 0$ for $x\in \R^N\setminus B_R(0)$, it follows  that
\begin{align*}
|(I-\Delta)^{\log}\phi(x)|&\le  2d_NA\int_{\supp \ \phi}\frac{\omega(|x-y|)}{|x-y|^N}\ dy\le C_NA\int_{\supp \ \phi}\frac{e^{-|x-y|}}{|x-y|^{\frac{N+1}{2}}}\ dy\\
& \le C_NA\int_{\supp \ \phi}\frac{e^{-\frac{|x|}{2}}}{(1+|x|)^{\frac{N+1}{2}}}\ dy\le C_N|\supp\ \phi|A\frac{e^{-\frac{|x|}{2}}}{(1+|x|)^{\frac{N+1}{2}}}.
\end{align*}
Therefore, combining the above computations, we find that 
\[
|(I-\Delta)^{\log}\phi(x)|\le C_{N,\phi}A\frac{e^{-|x|}}{(1+|x|)^{\frac{N+1}{2}}}\qquad \text{ for all }   x\in \R^N.
\]
From the above computations, we have that $|\langle \logrel u,\phi\rangle|\le C_{N,\phi}\|\phi\|_{C^{\alpha}(\R^N)}\|u\|_{L_0(\R^N)}$ and if  the sequence $\{u_n\}_n$ converges to $u$ in $L_0(\R^N)$ as $n\to \infty$ then 
\[
|\langle \logrel u_n-\logrel u,\phi\rangle|\le C_{N,\phi}A\|u_n-u\|_{L_0(\R^N)}\to 0 \quad\text{ as } n\to \infty.
\]
Proof of $(iii)$. This follows from $(i)$ and  the inequality
\[
|2u(x)-u(x+y)-u(x-y)|\le \|u\|_{C^{\alpha}(B_r(0))}|y|^{\alpha}\qquad \text{ for } \ y\in B_{r/2}(0).
\]
Proof of $(iv)$. Let $0<r<1$ be small.  We have the following  estimate  of the difference,
\begin{align*}
|(I-\Delta)^{\log}u(x_1)-(I-\Delta)^{\log}u(x_2)|\le d_N(I_1+I_2)
\end{align*}
where $I_1$ and $I_2$ are given by
\begin{align*}
I_1& := \int_{B_r} \frac{|u(x_1)-u(x_1+y)|+|u(x_2)-u(x_2+y)|}{|y|^N}\omega(|y|)\ dy\\\
I_2& := \int_{\R^N\setminus B_r} \frac{|u(x_1)-u(x_2)|+|u(x_1+y)-u(x_2+y)|}{|y|^N}\omega(|y|)\ dy
\end{align*}
For $I_1$, we use the inequality $|u(x_1)-u(x_1+y)|\le\|u\|_{C^{\beta}(\R^N)} |y|^{\beta}$  to get
\[
I_1\le2\|u\|_{C^{\beta}(\R^N)}\int_{B_r}|y|^{\beta-N}\omega(|y|)\ dx\le  \frac{2\omega_{N-1}\Gamma(N/2)}{\beta}\|u\|_{C^{k}(\R^N)}r^{\beta}
\]
For $I_2$, we use $|u(x_1)-u(x_2)|+|u(x_1+y)-u(x_2+y)|\le 2\|u\|_{C^{\beta}(\R^N)}|x_1-x_2|^{\beta}$ and, 
\begin{align*} 
I_2&\le 2|x_1-x_2|^{\beta}\|u\|_{C^{\beta}(\R^N)}\Big(\int_{B_1\setminus B_r}\frac{\omega(|y|)}{|y|^{N}}\ dy + \int_{\R^N\setminus B_1}\frac{\omega(|y|)}{|y|^{N}}\ dy\Big)\\
&\le 2|x_1-x_2|^{\beta}\|u\|_{C^{\beta}(\R^N)}\Big(\Gamma(N/2)\int_{B_1\setminus B_r}\frac{1}{|y|^N}\ dy+ \int_{\R^N\setminus B_1}{\frac{e^{-|y|}}{|y|^{\frac{N+1}{2}}}}\ dy\Big)\\
&\le 2|x_1-x_2|^{\beta}\|u\|_{C^{\beta}(\R^N)}\omega_{N-1}\Big(\Gamma(N/2)\log \frac{1}{r} + \Gamma(N,1)\Big)\\
&\le 2|x_1-x_2|^{\beta}\omega_{N-1}\|u\|_{C^{\beta}(\R^N)}\Big(\frac{\Gamma(N/2)r^{-\epsilon}}{\epsilon} + \Gamma(N,1)\Big),
\end{align*}
where we have used the inequality $\log(\rho)\le \frac{\rho^{\epsilon}}{\epsilon}$ for $\epsilon>0$ and $\rho\ge 1$ (see \cite{RefSAT}). Therefore, taking $r=|x_1-x_2|$, we ends with
 \[
 |(I-\Delta)^{\log}u(x_1)-(I-\Delta)^{\log}u(x_2)|\le C(N,\beta,\epsilon)\|u\|_{C^{\beta}(\R^N)}|x_1-x_2|^{\beta-\epsilon}.
 \]
 Proof of $(v)$. This easily follows by integrating the following equality
 \[
 (\phi(x)\psi(x)-\phi(y)\psi(y)) = (\phi(x)-\phi(y))\psi(x)+(\psi(x)-\psi(y))\phi(x)-(\phi(x)-\phi(y))(\psi(x)-\psi(y)),
 \]
 while the second statement is an application of Fubini's theorem.
 This completes the proof of Proposition \ref{well-defined}.
\end{proof}

 We next list  some properties for functions belonging to the space  $H^{\log}(\R^N)$.
\begin{lemma}\label{Embedd}
The following assertions  hold true
\begin{itemize}
\item[1.] If $u\in H^{\log}(\R^N)$, then  $|u|,u^{\pm}\in H^{\log}(\R^N)$ with $\||u|\|_{H^{\log}(\R^N)},\|u^{\pm}\|_{H^{\log}(\R^N)}\le \|u\|_{H^{\log}(\R^N)}$.
\item[2.]  The space $\cC_c^{0,\alpha}(\R^N)\subset  H^{\log}(\R^N)$ for any $\alpha>0$.
\item[3.] If $\phi\in \cC_c^{0,\alpha}(\R^N)$  and $u\in H^{\log}(\R^N)$, then $\phi u\in H^{\log}(\R^N)$ and there a constant $C:=C(N,\phi)>0$ such that 
 \[
 \|\phi u\|^2_{H^{\log}(\R^N)}\le C\|u\|^2_{H^{\log}(\R^N)}\qquad 
 \]
\end{itemize}
\end{lemma}

 \begin{proof}
 It straightforward to see by integrating  the inequality $$||u(x)|-|u(y)||\le |u(x)-u(y)|$$ that $\cE_{\g}(|u|,|u|)\le \cE_{\g}(u,u,)$ and  $\||u|\|_{H^{\log}(\R^N)}\le \|u\|_{H^{\log}(\R^N)}$. Using also the inequality 
 $$2(u^+(x)-u^+(y))(u^{-}(x)-u^{-}(y))=- 2(u^{-}(x)u^{+}(y)+u^{-}(y)u^{+}(x))\le 0\quad \text{ for } \ x,y\in \R^N,$$
  it follows  that
 \[
 \cE_{\g}(u,u)= \cE_{\g}(u^{+},u^{+})+\cE_{\g}(u^{-},u^{-}) -2\cE_{\g}(u^{+},u^{-})\ge \cE_{\g}(u^{+},u^{+})+\cE_{\g}(u^{-},u^{-}),
 \]
 proving clearly that  the first item holds.  Now, for the second item, we let $u\in \cC_c^{0,\alpha}(\R^N)$ be such that $\supp \ u \subset B_r$, $r>0$. without loss of generality we may assume that $0<r<1$ such that we can directly apply the asymptotics in \eqref{Asymptoti-o}. We therefore have
 \begin{align*}
  \cE_{\g}(u,u)&=\frac{1}{2}\int_{B_r}\int_{B_r}|u(x)-u(y)|^2J(x,y) \ dx dy   +\int_{B_r}u^2(x)\int_{\R^N\setminus B_r}J(x-y)\ dydx\\
  &\le C_1\int_{B_r}\int_{B_r}|x-y|^{2\alpha-N}   dx dy +C_2\int_{B_r}u^2(x)\big(\int_{B_1\setminus B_r}|x-y|^{-N}\ dy\\
  &\qquad\qquad\qquad+\int_{\R^N\setminus B_1}e^{-|x-y|} dy\big)dx\le C\frac{|B_r(0)|}{2\alpha}r^{-2\alpha}+C_3,
 \end{align*}
 where the constants $C:=C(N)>0$, $C_2:=C_2(r,N)>0$ and $C_3:=C_3(r,N)>0$. The second item is proved. We next prove item 3. Let $u\in H^{\log}(\R^N)$ and $\phi\in \cC_c^{0,\alpha}(\R^N)$ with $\supp \ \phi \subset B_r$, for $0<r<1$ . Then using the inequality $$|\phi(x)u(x)-\phi(y)u(y)|^2\le 2(|u(x)-u(y)|^2|\phi(x)|^2+|u(y)|^2|\phi(x)-\phi(y)|^2),$$ we get
 \begin{align*}
  \cE_{\g}(u,u)&\le \int_{B_r}\int_{B_r}|\phi(x)|^2|u(x)-u(y)|^2J(x,y)  dx dy \\
   &\hspace{2cm} +2\int_{B_r}u^2(x)\int_{ B_r}|\phi(x)-\phi(y)|^2J(x-y) dydx\\
  &\hspace{2cm} +C\int_{B_r}|\phi(x)u(x)|^{2}\big(\int_{B_1\setminus B_r}|x-y|^{-N}\ dy+\int_{\R^N\setminus B_1}e^{-|x-y|} dy\big)dx\\
  &\le 2\|\phi\|^2_{L^{\infty}(\R^N)} \cE_{\g}(u,u) +C_2\int_{B_r}u^2(x)\int_{ B_r}|x-y|^{2\alpha-N } dydx+C_3<\infty.
 \end{align*}
 Since $\|\phi u\|_{L^2(\R^N)}\le C_{\phi}\|u\|_{L^2(\R^N)}$, we have that $\phi u\in H^{\log}(\R^N)$ and item 3 is proved. 
 \end{proof}
We recall the  space $\cH^0_0(\Omega)$, corresponding to the analytical framework for the logarithmic  Laplacian $\loglap$ introduced in \cite{RefHT}, see also \cite{RefPST},  given by
\begin{equation}
\cH_0^0(\Omega)=\Big \{u\in L^2(\R^N):u\equiv0  \text{ on } \Omega^c \text{ and}\iint_{\substack{ x,y\in \R^N\\|x-y|<1}}\frac{|u(x)-u(y)|^2}{|x-y|^N}dxdy<\infty\Big\}.
\end{equation}
Here $\Omega^c=\R^N\setminus\Omega$, and the map 
\[
(u,v)\mapsto\langle u,v\rangle_{\cH^0_0(\Omega)}:=\frac{C_N}{2}\iint_{\substack{ x,y\in \R^N\\|x-y|<1}}\frac{(u(x)-u(y))(v(x)-v(y))}{|x-y|^N}\ dxdy,
\]
is a scalar product on $\cH^0_0(\Omega)$. The space  $\cH^0_0(\Omega)$ is a Hilbert space with induced norm $\|\cdot\|_{ \cH^0_0(\Omega)}=\langle\cdot,\cdot\rangle^{\frac{1}{2}}_{\cH_0^0(\Omega)}$. Moreover,  The space $C_c^{2}(\Omega)$ is dense in  $\cH^0_0(\Omega)$ and 
\begin{equation*}\label{emb1}
\text{the embedding \  $\cH^0_0(\Omega) \hookrightarrow {L^2(\Omega)}$ \ is compact}.
\end{equation*}

We have the following Lemma

\begin{lemma}
\label{Properties}
\begin{itemize}
\item[(i)] the space $H^{\log}(\R^N)$ is a Hilbert space and,  $H^{m}(\R^N)\subset H^{\log}(\R^N)$ for all $m>0$.
\item[(ii)] If $\Omega\subset \R^N$ is an open set  with  finite measure then we have the following  Poincar\'e inequality \text{ with } \  $C:=C(N,\Omega)$
 \begin{equation}
\|u\|^2_{L^2(\Omega)}\le C \int_{\R^N}\int_{\R^N}|u(x)-u(y)|^2J(x-y)  dx dy, \qquad u\in \cH_0^{\log}(\Omega) 
\end{equation}
\item[(iii)] If $\Omega\subset \R^N$ is bounded, then there a constant $ C_j:=C(N,\Omega)$, $j=1,2$ such that 
 $$C_1\cE_{\g}(u,u)\le\|u\|^{2}_{\H_0^{0}(\Omega)}\le C_2\cE_{\g}(u,u)$$
 \item[(iv)]  The space $\cC^{\infty}_c(\Omega)$ is dense in $\H_0^{\log}(\Omega)$ and 
 \begin{equation}\label{emb}
\text{the embedding \  $\cH^{\log}_0(\Omega) \hookrightarrow {L^2(\Omega)}$ \ is compact}.
\end{equation}
 \end{itemize}
\end{lemma}	
\begin{proof}
Let $\{u_n\}_n\subset H^{\log}(\R^N)$ be a Cauchy sequence. Then  $\{u_n\}_n$ is in particular a Cauchy sequence in $L^2(\R^N) $ and
hence there exists a  $u\in L^2(\R^N)$ such  that $u_n\to u$ as $n\to \infty$. Passing to a subsequence we get that $u_n\to u$ a.e in $\R^N$ as $n\to \infty$ and by Fatou Lemma we have
\[
\cE_{\g}(u,u)\le \liminf_{n\to \infty}\cE_{\g}(u_n,u_n)\le \sup_{n\in \N}\cE_{\g}(u_n,u_n)<\infty,
\]
showing that $u\in H^{\log}(\R^N)$. Apply once more Fatou Lemma it follows that
\begin{align*}
\|u_n-u\|^2_{H^{\log}(\R^N)}&=\|u_n-u\|^2_{L^2(\R^N)}+ \cE_{\g}(u_n-u,u_n-u)\le  \liminf_{n\to \infty}\|u_n-u_m\|^2_{H^{\log}(\R^N)},
\end{align*}
 for $ n,m\in \N$. The claim follows since $\{u_n\}_n$ is a Cauchy sequence in $H^{\log}(\R^N)$. 

By Plancherel thereon the norm in $H^{\log}(\R^N)$ is also given via Fourier representation
 \[
 \|u\|_{H^{\log}(\R^N)}= \Big(\|u\|^2_{L^2(\R^N)}+\int_{\R^N}\log(1+|\xi|^2)|\cF(u)(\xi)|^2\ d\xi\Big)^{\frac{1}{2}}.
 \]
 Threfore, using the standard inequality $\log\rho\le \frac{\rho^m}{m}$ for $\rho\ge 1$ for $m>0$ (see  e.g.\cite{RefSAT})  one see that the space $H^{\log}(\R^N)$ is larger than any  Sobolev space $H^m(\R^N):= W^{m,2}(\R^N)$. In fact if $u\in H^{m}(\R^N)$ then the proof of $(i)$ is completed by the following inequality,
 \begin{equation}\label{classical}
 \begin{split}
 \|u\|^2_{H^{\log}(\R^N)}&=\|u\|^2_{L^2(\R^N)}+ \int_{\R^N}\log(1+|\xi|^2)|\cF(u)(\xi)|^2 d\xi\\
 &\le \|u\|^2_{L^2(\R^N)}+ \frac{1}{m}\int_{\R^N}(1+|\xi|^2)^m|\cF(u)(\xi)|^2 d\xi\le C_m\|u\|^2_{H^{m}(\R^N)}.
 \end{split}
 \end{equation}
 The Poincar\'e inequality in $(ii)$ follows from \cite[Lemma 2.9]{RefMMP}  and \cite{RefST} if $\Omega$ is bounded or bounded in one direction. We provide the proof here for $\Omega\subset \R^N$ with $|\Omega|<\infty$.   Since $u= 0$ in $\R^N\setminus \Omega$,  we first have by  H\"older  inequality that 
\[
|\hat u(\xi)|^2\le (2\pi)^{-N}|\Omega|\|u\|^2_{L^2(\Omega)}\qquad\text{ for every }\quad \xi\in \R^N.
\]
Next, by Plancherel theorem  and every $R>0$,  we get
 \begin{align*}
\|u\|^2_{L^2(\Omega)}& = \int_{\R^N}|\hat u(\xi)|^2\ d\xi 
= \int_{|\xi|<R}|\hat u(\xi)|^2\ d\xi+ \int_{|\xi|\ge R}\log(1+|\xi|^2)|\hat u(\xi)|^2\log(1+|\xi|^2)^{-1}\ d\xi\\
&\le (2\pi)^{-N}R^{N}|\Omega||B_1(0)| \|u\|^2_{L^2(\R^N)}+\frac{1}{2\log(1+R^{2})}\int_{\R^N}\int_{\R^N}(u(x)-u(y))^2J(x-y)\ dxdy.
\end{align*}
Therefore, choosing $R<2\pi (|\Omega||B_1(0)|)^{-\frac{1}{N}}=  2\pi\Big(\frac{N}{\omega_{N-1}|\Omega|}\Big)^{\frac{1}{N}}$ we find that 
\begin{align*}
\|u\|^2_{L^2(\Omega)}&\le \frac{2}{\log(1+R^2)\Big(1- (2\pi)^{-N}R^{N}|\Omega||B_1(0)|\Big)}\int_{\R^N}\int_{\R^N}(u(x)-u(y))^2J(x-y)\ dxdy.
\end{align*}
The  proof of $(ii)$ follows here by minimizing in $R$ the coefficient in the right hand side.

 For item  $(iii)$, we use  the asymptotics in \eqref{Asymptoti-o}  to get
\begin{align*}
 \|u\|^2_{\cH_0^0(\Omega)}= \frac{1}{2}\iint_{\substack{ x,y\in\R^N\\|x-y|<1}}\frac{|u(x)-u(y)|^2}{|x-y|^N}\ dxdy &\le C_1 \int_{\R^N}\int_{\R^N}\frac{|u(x)-u(y)|^2}{|x-y|^N}\omega(|x-y|)\ dxdy.
\end{align*}
Next,  using Poincar\'{e} inequality for $\cH_0^0(\Omega)$ again with \eqref{Asymptoti-o}  we get that
\begin{align*}
 &\cE{\g}(u,u)=  \frac{d_N}{2}\int_{\R^N}\int_{\R^N}\frac{|u(x)-u(y)|^2}{|x-y|^N}\omega(|x-y|)\ dxdy\\
 & \hspace{1cm}\le \Gamma(\frac{N}{2})  \iint_{\substack{ x,y\in \Omega\\|x-y|<1}}\frac{|u(x)-u(y)|^2}{|x-y|^N}\ dxdy+ 2\int_{\Omega}|u(x)|^2 \int_{ \Omega\cap \{|x-y|\ge 1\}}\omega(|x-y|)\ dydx\\
& \hspace{3cm}+\int_{\Omega}|u(x)|^2\int_{\R^N\setminus\Omega}\frac{\omega(|x-y|)}{|x-y|^N}\ dydx \  \le C_2\|u\|_{\cH_0^0(\Omega)}
\end{align*}
with 
\[
C_2:=C\Big(1+ \sup_{x\in\Omega }\Big(\int_{\R^N\setminus\Omega}\frac{\omega(|x-y|)}{|x-y|^N}\ dy +\int_{ \Omega\cap \{|x-y|\ge 1\}}\omega(|x-y|)\ dy\Big)\Big)<\infty.
\]
The proof of $(iv)$ follows from \cite[Theorem 3.1]{RefHT}  and $(iii)$  since The space $C_c^{\infty}(\Omega)$ is dense in  $\cH^0_0(\Omega)$ and 
$
\text{the embedding \  $\cH^0_0(\Omega) \hookrightarrow {L^2(\Omega)}$ \ is compact}.
$
The  proof ends here.
\end{proof}
As consequence of the Poincar\'e inequality, we have for  bounded $\Omega$  with continuous boundary that  the space $\cH_0^{\log}(\Omega)$ can be identified by
 \[
\cH_0^{\log}(\Omega)=\Big \{ u\in H^{\log}(\R^N): \quad u\equiv 0 \text{ \ on \ }\ \R^N\setminus \Omega\Big\}.
\]
 and it is a Hilbert  space endowed with the scalar product
$
	(v,w) \mapsto \cE_{\g}(v,w) 
$
and the corresponding norm ~$\|u\|_{\H_0^{\log}(\Omega)}=\sqrt{\cE_{\g}(u,u)}$.

%

\section{Eigenvalue problem }\label{Sec-eigenvalue-problem}
In this section, we provide the proof of Theorem \ref{eigenvalue},   proposition \ref{Bound-eigenfunction}  and Theorem \ref{Faber-Krahn} concerning the study of the Dirichlet eigenvalue problem in bounded open set $\Omega$, 
\begin{equation}\label{eq2-eigenvalue}
	\begin{split}
	\quad\left\{\begin{aligned}
		(I-\Delta)^{\log}u &= \lambda u && \text{ in $\Omega$}\\
		u &=0             && \text{ on }  \mathbb{R}^N\setminus \Omega. 
	\end{aligned}\right.
	\end{split}
	\end{equation}
We start with the
\begin{proof}[Proof of Theorem \ref{eigenvalue}]
Let $\Psi: \H_0^{\log}(\Omega)\to \R$ be the functional defined by 
\[
\Psi(u):= \cE_{\g}(u,u)=\|u\|^2_{\H_0^{\log}(\Omega)}. \quad 
\]
We use the direct method of minimization. Let $\{u_n\}_{n\in \N}$  be a minimizing sequence for $\Psi$ in $\cP_1(\Omega):=\{u\in \cH^{\log}_0(\Omega): \|u\|_{L^2(\Omega)}=1\}$, that is
\[
\lim_{n\to \infty}\Psi(u_n) = \inf_{u\in \cP_1(\Omega)}\Psi(u)\ge 0>-\infty.
\]
Then by the definition of $\Psi$, the sequence $\{u_n\}_{n\in \N}$ is bounded in $\H_0^{\log}(\Omega)$ and up to subsequence, there  exists $u_0\in \H_0^{\log}(\Omega)$ such that thanks to \eqref{emb},
\begin{align}
u_n\rightharpoonup u_0 \quad\text{ weakly  }\quad  \text{  in } \H_0^{\log}(\Omega)\\\label{L-2-conv}
u_n \to u_0 \quad\text{ strongly  }\quad  \text{ in }  L^2(\Omega).
\end{align} 
It follows from \eqref{L-2-conv} that $\|u_0\|_{L^2(\Omega)}= 1$ and that $u_0\in \cP_1(\Omega)$. Using the lower-semi-continuity of the norm in $ \H_0^{\log}(\Omega)$, we deduce that
\[
\inf_{u\in \cP_1(\Omega)}\Psi(u)= \lim_{n\to \infty}\Psi(u_n)\ge \Psi(u_0)\ge \inf_{u\in \cP_1(\Omega)}\Psi(u).
\]
This yields that $ \Psi(u_0) = \displaystyle\inf_{u\in \cP_1(\Omega)}\Psi(u)$ and, the first eigenvalue is $\lambda_1(\Omega) = \Psi(u_0)$, with the corresponding eigenfunction $\phi_1= u_0\in \cP_1(\Omega)$. By the Lagrange multipliers theorem, there exists $\lambda\in \R$ such that 
\begin{equation}\label{Lagrange-constant}
\cE_{\g}(\phi_1,v)=\langle\Psi'(\phi_1),v\rangle= \lambda\int_{\Omega}\phi_1v\ dx\quad \text{ for all } \quad v\in \H_0^{\log}(\Omega).
\end{equation}
Taking in particular $v=\phi_1$, we find that $\lambda=\lambda_1(\Omega)=\cE_{\g}(\phi_1,\phi_1)$. We next show that $\phi_1$ does not change sign in $\Omega$. Indeed, since $\cE_{\g}(|v|,|v|)\le \cE_{\g}(v,v)$ for  $v\in \H_0^{\log}(\Omega)$, it follows that   $|\phi_1|\in \cP_1(\Omega)$ and by the definition of $\lambda_1(\Omega)$ we have  that 
\[
\lambda_{1}(\Omega)= \cE_{\g}(|\phi_1|,|\phi_1|),
\]
showing that $\phi_1$ does not change sign in $\Omega$. We may assume that $\phi_1$ is nonnegative.  Suppose then that $\phi_1(x_0) = 0$ for some $x_0\in\Omega$. Then
\[
0=\lambda_1(\Omega)\phi_1(x_0) =-d_N\int_{\R^N}\frac{\phi_1(x_0)}{|x-y|^N}\g(|x-y|)\ dy<0
\]
which yields a contradiction. Therefore $\phi_1>0$ in $\Omega$ and $(i)$ is proved.

We prove $(ii)$ via contradiction. Suppose that there exists a function $v\in \cP_1(\Omega)$  satisfying $(I-\Delta)^{\log}v=\lambda_1 v$ with $v\neq \alpha \phi_1$ for every $\alpha\in \R$. Then $w:= v-\alpha\phi_1$ satisfies also $(I-\Delta)^{\log}w=\lambda_1 w$. But since $\phi_1>0$ in $\Omega$, by choosing  $\alpha=\frac{v(x_0)}{\phi_1(x_0)}$, $x_0\in \Omega$, it follows that $w$ vanishes at  $x_0\in \Omega$ and therefore must change sign. This contradicts  $(i)$ and thus the eigenvalue $\lambda_1(\Omega)$ is simple.

We prove $(iii)$ by  induction. We first note that, if follows from  the simplicity of $\lambda_{1}(\Omega)$ in $(ii)$ that $\lambda_1(\Omega)<\lambda_2(\Omega)$. By the same construction as in the case $k=1$, we get a sequence of  eigenfunctions $\phi_{2},\cdots,\phi_{k}\in \cH^{\log}_0(\Omega)$ and eigenvalues $\lambda_{2}(\Omega)\le \cdots\le \lambda_{k}(\Omega)$, $k\in\N$ with the properties that 
\[
\lambda_{j}(\Omega)= \inf_{u\in\cP_j(\Omega)}\cE_{\g}(u,u)= \cE_{\g}(\phi_j,\phi_j), \qquad j=1,\cdots, k \quad\text{ and }
\]
\[
\cE_{\g}(\phi_j,v)= \lambda_j(\Omega)\int_{\Omega}\phi_jv\ dx\quad \text{ for all } \quad v\in \H_0^{\log}(\Omega).
\]
Next, we define $\lambda_{k+1}(\Omega)$ as in \eqref{Lambda-k}, that is 
\[
\lambda_{k+1}(\Omega)= \inf_{u\in \cP_{k+1}(\Omega)}\cE_{\g}(u,u).  
\]
By the same argument as above, the value $\lambda_{k+1}(\Omega)$ is attained by a function $\phi_{k+1}\in \cP_{k+1}(\Omega)$  and by the Lagrange multipliers theorem, there exists $\lambda\in \R$ such that 
\begin{equation}\label{Lagrange-constant-for -k}
\cE_{\g}(\phi_{k+1},v)=  \lambda\int_{\Omega}\phi_{k+1}v\ dx\quad \text{ for all } \quad v\in \cP_{k+1}(\Omega).
\end{equation}
Taking in particular $v=\phi_{k+1}$ in \eqref{Lagrange-constant-for -k}, we get that $\lambda=\lambda_{k+1}(\Omega)$. Moreover, for $j=1,\cdots k$, it follows from the definition of $\cP_{k+1}(\Omega)$ and  taking $v=\phi_{j}$ in \eqref{Lagrange-constant-for -k}, we find that
\begin{equation}\label{equal-0}
\cE_{\g}(\phi_{k+1},\phi_{j})= 0 = \lambda_j(\Omega)\int_{\Omega}\phi_{k+1}\phi_{j}\ dx.
\end{equation}
In other to conclude that $\phi_{k+1}$ is an eigenfunction corresponding to eigenvalue $\lambda_{k+1}(\Omega)$, we need to show that \eqref{Lagrange-constant-for -k} holds for all $v\in \cH^{\log}_0(\Omega)$. To see this we write $\cH^{\log}_0(\Omega)=span\{\phi_1,\cdots,\phi_k\}\oplus\cP_{k+1}(\Omega)$ such that any $v\in \cH^{\log}_0(\Omega)$ can be written as $v=v_1+v_2$ with $v_1\in span\{\phi_1,\cdots,\phi_{k}\}$ and $v_2\in \cP_{k+1}(\Omega)$. It follows from \eqref{Lagrange-constant-for -k} with $v$ replaced by $v_2=v-v_1\in \cP_{k+1}(\Omega)$ that 
\begin{align*}
0&=\cE_{\g}(\phi_{k+1}, v_2)-\lambda_{k+1}(\Omega)\int_{\Omega}\phi_{k+1}v_2\ dx\\
&=\cE_{\g}(\phi_{k+1}, v)-\cE_{\g}(\phi_{k+1}, v_1)-\lambda_{k+1}(\Omega)\int_{\Omega}\phi_{k+1}(v-v_1)\ dx\\
&=\cE_{\g}(\phi_{k+1}, v)-\lambda_{k+1}(\Omega)\int_{\Omega}\phi_{k+1}v\ dx,
\end{align*}
where we used  equality in \eqref{equal-0}. This shows that \eqref{Lagrange-constant-for -k} holds for all $v\in \cH^{\log}_0(\Omega)$. We have just  constructed inductively an $L^2$-normalized sequence $\{\phi_{k}\}_{k\in\N}$ in  $\cH^{\log}_0(\Omega)$ and a nondecreasing sequence  $\{\lambda_{k}\}_{k\in \N}$ in $\R$  such that \eqref{Lambda-k} holds and such that $\phi_k$ is an eigenfunction of \eqref{eq1-eigenvalue} corresponding to $\lambda=\lambda_{k}(\Omega)$ for every $k \in  \N$. Moreover, we have by construction that $\{\phi_{k}\}_{k\in\N}$ form an orthogonal system in $L^2(\Omega)$. To complete the proof of $(iii)$, it remains to show that $\lim_{k\to +\infty}\lambda_{k}(\Omega)=+\infty$. Suppose by contradiction that $$\cE_{\g}(\phi_k,\phi_k)=\lambda_k(\Omega)\to c_0\in \R \quad\text{ as } \quad k\to +\infty \qquad\text{ for every }\ k\in \N.$$
Then the sequence  $\{\phi_k\}_{k\in\N}$ is bounded in $\cH^{\log}_0(\Omega)$ and, up to subsequence, there is $\phi_0\in \cH^{\log}_0(\Omega)$  such that 
\[
\phi_k\to \phi_0 \quad\text{ in } \quad L^2(\Omega) \quad\text{ as } k\to +\infty.
\]
It follows in particular that $\{\phi_k\}_{k\in\N}$ is a Cauchy sequence in $L^2 (\Omega)$. But orthogonality in $L^2(\Omega)$ implies that $\|\phi_k-\phi_j\|_{L^2(\Omega)}=2$ for every $k$ and $j$, which leads to a contradiction.

For the proof of assertion $(iv)$, the orthogonality follows from from $(iii)$. we  then need to show that the sequence of eigenfunctions $\{\phi_k\}_{k\in\N}$ is a basis for both $L^2(\Omega)$ and $\cH^{\log}_0(\Omega)$. Let suppose by contradiction that there exists a nontrivial $u\in \cH_0^{\log}(\Omega)$ with 
\begin{equation}\label{Contradiction}
\text{$\|u\|_{L^2(\Omega)}=1$ \    and \  $\int_{\Omega}\phi_k u\ dx =0$ \  for any $k\in \N$ }.
\end{equation}
Since we have that $\displaystyle\lim_{k\to +\infty}\lambda_{k}(\Omega)= +\infty$, there exists an integer $k_0>0$ such that 
\[
\Psi(u) <\lambda_{k_0}(\Omega)= \inf_{v\in \cP_{k_0}(\Omega)}\Psi(v).
\]
This implies that $u\notin \cP_{k_0}(\Omega)$ and, by the definition of $ \cP_{k_0}(\Omega)$, we have  that $\int_{\Omega}\phi_ju \ dx \neq 0$ for some $j\in\{1,\cdots,k_0-1\}$. This contradicts \eqref{Contradiction}. We conclude that $\cH^{\log}_0(\Omega)$ is contained in the $L^2$-closure of the span of $\{\phi_{k} : k \in  \N\}$. Since $\cH^{\log}_0(\Omega)$ is dense in $L^2(\Omega)$, we conclude that the span of $\{\phi_k : k\in\N\}$ is dense in $L^2(\Omega)$, and hence, the sequence  $\{\phi_k\}_{ k\in\N}$ is an orthonormal basis of $L^2(\Omega)$. This complete the proof of Theorem \ref{eigenvalue}.
\end{proof}
We next give the

\begin{proof}[Proof of Proposition \ref{Bound-eigenfunction}]

We work here with the $\delta$-decomposition of the nonlocal operators as described in \cite[Theorem 3.1]{RefPST}. For this, let $\Omega\subset \R^N$ be  open and bounded set of $\R^N$. For $\delta>0$, we let $J_{\delta}:= 1_{B_{\delta}}J$ and $K_{\delta}:= J-J_{\delta}$.  Note that for $u,v\in \cH^{\log}_0(\Omega)$, 
\begin{equation*}\label{split1-bilinear}
\begin{split}
\cE_{\g}(u,v) &=\cE^{\delta}_{\g}(u,v)+  \frac{d_N}{2}\int_{\R^N}\int_{\R^N}(u(x)-u(y))(v(x)-v(y))K_{\delta}(x-y) \ dxdy \\
&=\cE^{\delta}_{\g}(u,v)+\kappa_{\delta}\langle u,v\rangle_{L^2(\R^N)}-\langle K_{\delta}\ast u,v\rangle_{L^2(\R^N)}
\end{split}
\end{equation*}
where the $\delta$-dependent quadratic form $\cE^{\delta}_{\g}$ is given by 
\[
(u,v)\mapsto \cE_{\g}^{\delta}(u,v)=\frac{d_N}{2}\int_{\R^N}\int_{\R^N}(u(x)-u(y))(v(x)-v(y))J_{\delta}(x-y)  \ dxdy, 
\]
the function $K_{\delta} \in L^1(\R^N)$  and the constant $\kappa_{\delta}$ is 
\[
\kappa_{\delta}= \int_{\R^N}K_{\delta}(z)\ dz>\int_{B_1\setminus B_{\delta}} \frac{1}{|z|^{N}}\ dz = -c_N\ln\delta\to +\infty \quad\text{ as }\quad \delta\to 0.
\]
Next, let $c>0$ be a constant to be chosen later.   Consider the function $w_c=(u-c)^+: \Omega \to \R$. Then $w_c\in \cH^{\log}_0(\Omega)$ by Lemma \ref{Embedd} see also \cite[Lemma 3.2]{RefJ2}. Moreover, for $x,y \in \R^N$ we have that $(u(x)-u(y))(w_c(x)-w_c(y)) \ge  (w_c(x)-w_c(y))^2$. Indeed,
  \begin{align*}
    &(u(x)-u(y))(w_c(x)-w_c(y)) = ([u(x)-c]-[u(y)-c])(w_c(x)-w_c(y))\\
&=[u(x)-c]w_c(x)+[u(y)-c]w_c(y)-[u(x)-c]w_c(y) -w_c(x)[u(y)-c]\\
    &=w_c^2(x)+w_c^2(y)-2 w_c(x)w_c(y)+ [u(x)-c]^- w_c(y)+ w_c(x)[u(y)-c]^-\\
    &\ge w_c^2(x)+w_c^2(y)-2 w_c(x)w_c(y) = (w_c(x)-w_c(y))^2.
  \end{align*}
This implies that   
\begin{align}
  \cE_{\g}^\delta(w_c,w_c)&= \frac{d_N}{2}\int_{\R^N}\int_{\R^N}(w_c(x)-w_c(y))^2J_{\delta}(x-y)\,dxdy\nonumber\\
  &\le   \frac{d_N}{2}\int_{\R^N}\int_{\R^N} (u(x)-u(y))(w_c(x)-w_c(y))J_{\delta}(x-y)\,dxdy\nonumber\\
  & = \cE_{\g}(u,w_c)-\kappa_{\delta} \langle u,w_c \rangle_{L^2(\Omega)} + \langle K_{\delta} * u\,,w_c \rangle_{L^2(\Omega)} \\
&\le \bigl(\lambda-\kappa_{\delta}\bigr) \langle u,w_c \rangle_{L^2(\Omega)} + \|K_{\delta} * u\|_{L^{\infty}(\R^N)}\langle 1\,,w_c \rangle_{L^2(\Omega)}.\nonumber
  \label{result2:eq1}
\end{align}
Note that  $\kappa_{\delta} \to +\infty$ as $ \delta\to 0$. Hence, we may fix $\delta>0$ such that  $\lambda+\kappa_{\delta}<-1$. Moreover, with this choice of $\delta$, together with the trivial inequality $u(x)w_c(x)\ge cw_c(x)$  for $x\in \Omega$,  we infer that 
\begin{equation}\label{Ineq-delta}
\begin{split}
 \cE_{\g}^\delta(w_c,w_c)&\le   \int_{\Omega} (\|K_{\delta} * u\|_{L^{\infty}(\R^N)}-c)w_c\ dx\\
 &\le   \int_{\Omega} (c_{N,\delta}\| u\|_{L^{2}(\R^N)}-c)w_c\ dx.
 \end{split}
\end{equation}
 The quantity $ c_{N,\delta}\| u\|_{L^{2}(\R^N)}$ is obtained  in the following computation using  H\"{o}der's (or Young's) inequality combined with the asymptotics in \eqref{Asymptoti-o},
\begin{align*}
\|k_{\delta} * u\|_{L^{\infty}(\R^N)}&\le   c_{N,\delta}\|u\|_{L^2(\R^N)} .
\end{align*}
We then  deduce from \eqref{Ineq-delta} with $c>c_{N,\delta}\|u\|_{L^2(\R^N)} $  that 
\begin{equation}\label{Ineq-delta-2}
0\le  \cE_{\g}^\delta(w_c,w_c)\le   0,
\end{equation}
which implies that $\cE_{\g}^\delta(w_c,w_c)=0$. Consequently, $w_c = 0$ in $\Omega$  by the Poincar\'{e}  type inequality. But then $u(x) \le c  \quad a.e.$  \  in $\Omega$, and therefore
\[
u(x)\le c_{N}\|u\|_{L^2(\R^N)}.
\]
Repeating the above argument for $-u$ in place of $u$ , we conclude that 
\[
\|u\|_{L^{\infty}(\Omega)}\le c\|u\|_{L^2(\R^N)}.
\]
This complete the proof of Proposition \ref{Bound-eigenfunction}.
\end{proof}
For the proof of Theorem \ref{Faber-Krahn}, we  first state a Polya-Szeg\"o type inequality  for $\logrel$.
\begin{lemma}\label{Polya-Szego}
Let $u^{\ast}$ be  the symmetric radial decreasing rearrangement of $u$.  Then,   
\begin{equation}\label{Energy-loglap}
\cE_{\g}(u^{\ast},u^{\ast})\le \cE_{\g}(u,u).
\end{equation}
Moreover, the  equality occurs for radial decreasing functions. Here,
\end{lemma}

\begin{proof}
By a changes of variable, we write the kernel  $J$ as
\begin{align*}
J(z)=d_N|z|^{-N}\omega(|z|)  
&= 4(\frac{\pi}{2})^{-\frac{N}{2}}\int_{0}^{\infty}e^{-t|z|^2}t^{\frac{N}{2}-1}e^{-\frac{1}{4t}}\ dt.
\end{align*}
 Then by Fubuni's theorem, we write the quadratic form as
\begin{align*}
\cE_{\g}(u,u)&= \frac{1}{2}\int_{\R^N}\int_{\R^N}|u(x)-u(y)|^2J(x,y)\ dxdy
= 2(\frac{\pi}{2})^{-\frac{N}{2}}\int_{0}^{\infty} G(t,u) \   t^{\frac{N}{2}-1}e^{-\frac{1}{4t}}\ dt,
\end{align*}
where, 
\[
G(t,u):= \int_{\R^N}\int_{\R^N}|u(x)-u(y)|^2e^{-t|x-y|^2}\ dxdy.
\]
Noticing that 
$$
\Big(e^{-t|z|^2}\Big)^{\ast} = e^{-t|z|^2}, \quad\text{ for all }\quad t\ge 0,
$$
It follows from \cite[corollary 2.3 and Theorem 9.2]{RefFJEH} see also \cite[Theorem $A_1$]{RefRLR} that 
\[
G(t,u^{\ast})\le G(t,u) \quad\text{ for all }\quad t\ge 0.
\]
This gives that 
\begin{equation}
\cE_{\g}(u^{\ast},u^{\ast})\le \cE_{\g}(u,u)\quad\text{ for  }\quad u\in H^{\log}(\R^N).
\end{equation}
The proof of Lemma \ref{Polya-Szego} is completed.
\end{proof}
\begin{proof}[Proof of Theorem \ref{Faber-Krahn}]
This   is a direct consequence of lemma \ref{Polya-Szego} and the characterization of the first eigenvalue $\lambda_{1,\log}(\Omega)$ of $\logrel$ in $\Omega$. 
 Since we know by Theorem \ref{eigenvalue} that the first eigenfunction $\phi_{1,\log}$ corresponding to $\lambda_{1,\log}(\Omega)$ is unique and strictly positive in $\Omega$, we   have  thanks to Lemma \ref{Polya-Szego}  that 
\[
\lambda_{1,\log}(\Omega)= \frac{\cE_{\g}(\phi_{1,\log},\phi_{1,\log})}{\|\phi_{1,\log}\|^2_{L^2(\Omega)}}\ge  \frac{\cE_{\g}(\phi^{\ast}_{1,\log},\phi^{\ast}_{1,\log})}{\|\phi^{\ast}_{1,\log}\|^2_{L^2(B^{\ast})}}\ge \inf_{u\in \cH_0^{\log}(B^{\ast})}\frac{\cE_{\g}(u,u)}{\|u\|^2_{L^2(B^{\ast})}}= \lambda_{1,\log}(B^{\ast}),
\] 
where we have used  (see \cite[Lemma 3.3]{RefFB}) the fact that
\[
\int_{\Omega}|u|^2\ dx = \int_{B^{\ast}} |u^{\ast}|^2 \ dx.
\]
This gives the proof of  \eqref{Faber}.  For the equality, if we suppose that $\lambda_{1,\log}(\Omega)= \lambda_{1,\log}(B^{\ast})$ with $|\Omega|= |B^{\ast}|$, then we must have the following  equality 
\[
\cE_{L}(\phi_{1,\log},\phi_{1,\log})=\cE_{L}(\phi^{\ast}_{1,\log},\phi^{\ast}_{1,\log})
\]
and by \cite[Lemma $A_2$]{RefRLR} we deduce that the first eigenfunction $\phi_{1,\log}$ has to  be proportional to a translate of a radially symmetric decreasing function such that the level set 
$$\Omega_0:=\{x\in \R^N:\quad \phi_{1,\log}>0\}$$
 is a ball. Since $\phi_{1,\log}>0$ in $\Omega$ by definition and it is unique, it follows  that $\Omega$ must coincide with $\Omega_0$ and has to be a ball. The proof of Theorem \ref{Faber-Krahn} is then completed.
\end{proof}

\section{Small order Asymptotics}\label{Sect-small}
This section is dedicated to the  proof of Theorem \ref{convergent-1}. We first  introduce some notions and preliminary lemmas that shall be used.
For $0< s<1$, we introduce the Sobolev space  (see\cite{RefEM,RefSAO})
\[
H^s(\R^N)= \Big\{ u\in L^2(\R^N): \quad \int_{\R^N}\int_{\R^N}\frac{|u(x)-u(y)|^2}{|x-y|^{N+2s}}\omega_{s}(|x-y|) dxdy<\infty\Big\}
\]
with corresponding norm given by 
\begin{align*}
\|u\|_{H^s(\R^N)}&=\Big( \|u\|^2_{L^2(\R^N)}+\int_{\R^N}\int_{\R^N}\frac{|u(x)-u(y)|^2}{|x-y|^{N+2s}}\omega_{s}(|x-y|) dxdy\Big)^{\frac{1}{2}}\\
&=\Big( \|u\|^2_{L^2(\R^N)}+\int_{\R^N}(1+|\xi|^2)^s|\cF(u)(\xi)|^2\, d\xi\Big)^{\frac{1}{2}}.
\end{align*}
Let $\Omega\subset\R^N$ be an open bounded set.   We will use the fact that (see \cite{RefEM})
$$\text{ the space  $C^2_c(\Omega)$ is dense in $\cH^s_0(\Omega)$,}$$  where the space $\cH_0^s(\Omega)$  is the completion of $\cC_c^{\infty}(\Omega)$ with respect to the norm $\|\cdot\|_{H^s(\R^N)}$. We start with the following Dirichlet eigenvalue problem
\begin{equation}\label{s-eigenvalue}
	\begin{split}
	\quad\left\{\begin{aligned}
		(I-\Delta)^{s}u &= \lambda u && \text{ in $\Omega$}\\
		u &=0             && \text{ on }  \mathbb{R}^N\setminus \Omega,
	\end{aligned}\right.
	\end{split}
	\end{equation}
	where $\Omega$ is a bounded Lipschitz  open set of $\R^N$. We define the first Dirichlet eigenvalue of $(I-\Delta)^s$ in $\Omega$ by 
\begin{equation}
\lambda_{1,s}(\Omega)= \inf_{u\in C^2_c(\Omega)}\frac{\cE_{\g ,s}(u,u)}{\|u\|_{L^2}(\Omega)}=\inf_{\substack{u\in C^2_c(\Omega)\\\|u\|_{L^2(\Omega)}=1}}\cE_{\g ,s}(u,u),
\end{equation}
where the quadratic form $(u,v)\mapsto\cE_{\g,s}(u,v)$ is defined by 
\begin{align*}
\cE_{\g,s}(u,v)&=\int_{\Omega}u(x)v(x)\ dx -\frac{d_{N,s}}{2}\int_{\R^N}\int_{\R^N}\frac{(u(x)-u(y))(v(x)-v(y))}{|x-y|^{N+2s}}\omega_{s}(|x-y|) dxdy\\
&=\int_{\R^N}(1+|\xi|^2)^s\cF(u)(\xi)\cF(v)(\xi)\, d\xi.
\end{align*}
By the Courant-Fischer minimax principle, the eigenvalues $\lambda_{k,s}(\Omega)$,  $k \in \N$ can be characterized equivalently as 
\begin{equation}\label{char-eigen}
\lambda_{k,s}(\Omega)= \inf_{\substack{V\subset \cH^s_0(\Omega)\\ \dim V = k}}\:\max_{\substack{v\in V\setminus\{0\} \\ \|v\|_{L^2(\Omega)}=1}} \cE_{\g,s}(v,v) =\inf_{\substack{V\subset C^2_c(\Omega)\\ \dim V = k}}\: \max_{\substack{v\in V\setminus\{0\} \\ \|v\|_{L^2(\Omega)}=1}} \cE_{\g,s}(v,v).
\end{equation}
\begin{remark}\label{Remark}
 Noticing that  $(1+|\xi|^2)^s\ge |\xi|^{2s}$ for $s\in (0,1)$ and $\xi\in \R^N$, we have via the  Fourier transform of the functional $\cE_{\g,s}(\cdot,\cdot)$ for $(I-\Delta)^s$ and $\cE_s(\cdot,\cdot)$ for the fractional Laplacian $(-\Delta)^s$ that 
$$\lambda_{k,s}(\Omega)= \cE_{\g,s}(\psi_{k,s},\psi_{k,s})\ge \cE_{s}(\psi_{k,s}, \psi_{k,s})\ge \inf_{\substack{v\in C_c^2(\Omega)\\\|v\|_{L^2(\Omega)}=1}}\cE_s(v,v)= \lambda^F_{1,s}(\Omega),$$
where $\psi_{k,s}$ is a $L^2$-normalized eigenfunction of $(I-\Delta)^s$ corresponding to $\lambda_{k,s}(\Omega)$ and  $\lambda_{1,s}^F(\Omega)$ is the first Dirichlet eigenvalue of the fractional Laplacian $(-\Delta)^s$ in $\Omega$  with
\[
\cE_{s}(u,v):=\frac{c_{N,s}}{2}\int_{\R^N}\int_{\R^N}\frac{(u(x)-u(y))(v(x)-v(y))}{|x-y|^{N+2s}}\ dxdy.
\]
\end{remark}
We need the following elementary estimates and inequalities.
\begin{lemma}
\label{elementary-est}
For $s \in (0,1)$ and $r>0$ we have 
\begin{equation}
\label{elementary-est-1}
\Bigl|\frac{(1+r^2)^{s}-1}{s}\Bigr| \le 2  \Bigl(1 + r^4 \Bigr)   
\end{equation}
and 
\begin{equation}
\label{elementary-est-2}
\Bigl|\frac{(1+r^2)^{s}-1}{s} -  \log(1+ r^2)\Bigr| \le 2s \Bigl( 1 + r^4 \Bigr). 
\end{equation}
Consequently, for every $u \in C^2_c(\Omega)$ and $s \in (0,1)$ we have 
\begin{equation}
  \label{eq:quadratic-form-est-eq1}
\Bigl|\cE_{\g,s}(u,u) - \|u\|_{L^2(\R^N)}^2 \Bigr| \le 2 s\Bigl( \|u\|_{L^2(\R^N)}^2 + \|\Delta u\|_{L^2(\R^N)}^2 \Bigr)  
\end{equation}
and 
\begin{equation}
  \label{eq:quadratic-form-est-eq2}
\Bigl|\cE_{\g,s}(u,u) - \|u\|_{L^2(\R^N)}^2 - s \cE_{\g}(u,u)\Bigr| \le 2 s^2 \Bigl(  \|u\|_{L^2(\R^N)}^2 + \|\Delta u\|_{L^2(\R^N)}^2 \Bigr)
\end{equation}
\end{lemma}
\begin{proof}
For fix $r>0$,  let $h_r(s)=(1+r^2)^{s}$. Then we have 
$$h_r'(\tau)=  (1+r^2)^{\tau} \ln (1+r^2)\text{ \quad  and  \quad  } h_r''(\tau)=  (1+r^2)^{\tau} \ln^2 (1+r^2). $$  Consequently, since  $(1+r^2)^{s} \le (1+r^2)$ for $s\in(0,1)$ and $\ln(1+r^2)\le (1+r^2) $,
$$
\Bigl|\frac{(1+r^2)^{s}-1}{s}\Bigr|= \frac{\ln(1+ r^2)}{s}\int_0^s (1+r^2)^{\tau}\,d\tau
\le  \ln(1+ r^2)(1+  r^{2})^s \le  2  \Bigl( 1+r^4 \Bigr)
$$
where in the last step we used that  $(1+r^2)^2\le 2(1+r^4)$ for $r>0$.
Hence \eqref{elementary-est-1}  holds. Moreover, by Taylor expansion,
$$
h_r(s)=  1 + s \ln (1+r^2) +  \ln^2(1+ r^2) \int_{0}^s (1+r^2)^{\tau}(s-\tau)d\tau
$$
and therefore 
$$
\Bigl|\frac{(1+r^2)^{s}-1}{s} -  \log (1+r^2)\Bigr| \le \frac{ \ln^2 (1+r^2)}{s}\Bigl|\int_{0}^s (1+r^2)^{\tau}(s-\tau)d\tau\Bigr| \le  s(1+ r^2)^{s}\ln^2(1+r^2) .
$$
But since $\ln^2(1+r^2)\le (1+r^2)$  and $(1+r^2)^s\le (1+r^2)$ for $s\in (0,1)$,  \eqref{elementary-est-2} follows.
Next, let $u \in C^2_c(\Omega)$ and $s \in (0,1)$. By \eqref{elementary-est-1} and Fourier transform for $\cE_{\g,s}$, we have 
\begin{align*}
 \Bigl|\cE_{\g,s}(u,u) - \|u\|_{L^2(\R^N)}^2\Bigr| &\le 
\int_{\R^N}\bigl|(1+ |\xi|^{2})^s-1 \bigr|\, |\hat u(\xi)|^2\,d\xi\\
&\le2s \int_{\R^N} \left(1+|\xi|^4\right) |\hat u(\xi)|^2\,d\xi\le 2s \Bigl( \| u\|_{L^2(\R^N)}^2  + \|\Delta u\|_{L^2(\R^N)}^2 \Bigr).
\end{align*}
Thus  \eqref{eq:quadratic-form-est-eq1} follows. Moreover, by~\eqref{elementary-est-2} we have  
\begin{align*}
 \Bigl|\cE_{\g,s}(u,u) - \|u\|_{L^2(\R^N)}^2 - s \cE_{\g}(u,u)\Bigr| &\le 
\int_{\R^N}\bigl| (1+|\xi|^{2})^s-1 - s \log(1+ |\xi|^2)\bigr|\, |\hat u(\xi)|^2\,d\xi\\
&= s\int_{\R^N}\Bigl| \frac{(1+|\xi|^{2})^s-1}{s} -  \log(1+ |\xi|^2)\Bigr|\, |\hat u(\xi)|^2\,d\xi\\
&\le 2s^2 \Bigl( \|u\|_{L^2(\R^N)}^2  + \|\Delta u\|_{L^2(\R^N)}^2 \Bigr).
\end{align*}
Hence \eqref{eq:quadratic-form-est-eq2} follows.  This completes the proof of Lemma \ref{elementary-est}.
\end{proof}

\begin{lemma}
\label{upper-est-lambda_k}
For all $k \in \N$ we have 
\begin{equation}
\label{uniform-upper-bound}
\lambda_{k,s}(\Omega) \le 1 +sC \qquad \text{for all $s \in (0,1)$}
\end{equation}
with a constant $C= C(N,\Omega,k)>0$,  and
\begin{equation}
  \label{eq:lim-inf-limsup-prelim}
\limsup_{s \to 0^+} \frac{\lambda_{k,s}(\Omega)-1}{s} \, \le \, \lambda_{k,\log}(\Omega).\end{equation}
Consequently, 
\begin{equation}
\label{eq:lim-lambda-k-to-1}
\lim \limits_{s \to 0^+} \lambda_{k,s}(\Omega)= 1\qquad\text{ for all $k \in \N$}. 
\end{equation}
\end{lemma}

\begin{proof}
We fix a subspace $V\subset C^2_c(\Omega)$ of dimension $k$ and let $S_V:=\{u\in V\;:\; \|u\|_{L^2(\Omega)}=1\}$. Using \eqref{char-eigen} and \eqref{eq:quadratic-form-est-eq1}, we find that, for $s \in (0,1)$,  
\begin{equation}
  \label{eq:upper-est-lambda_k-first-eq}
\frac{\lambda_{k,s}(\Omega)- 1}{s} \le \max\limits_{u\in S_V} \frac{\cE_{\g,s}(u,u) - 1}{s} \le C
\end{equation}
with 
$$
C = C(N,\Omega,k) = 2\max\limits_{u\in S_V}\Bigl( \|u\|_{L^2(\R^N)}^2 + \|\Delta u\|_{L^2(\R^N)}^2 \Bigr).
$$
Hence \eqref{uniform-upper-bound} holds. Moreover, setting $\cR_s(u)= \frac{\cE_{\g,s}(u,u) - 1}{s}- \cE_{\g}(u,u)$ for $u \in C^2_c(\Omega)$, we deduce from \eqref{eq:upper-est-lambda_k-first-eq} that 
$$ 
\frac{\lambda_{k,s}(\Omega)- 1}{s} \le \max\limits_{u\in S_V}\cE_{\g}(u,u) 
+ \max\limits_{u\in S_V}|\cR_s(u)| 
$$
while, by ~(\ref{eq:quadratic-form-est-eq2}), 
$$
|\cR_s(u)| \le 2 s \Bigl( \|u\|_{L^2(\R^N)}^2 + \|\Delta u\|_{L^2(\R^N)}^2\Bigr) \to 0 \quad \text{as $s \to 0^+$ uniformly in $u \in S_V$.}
$$
Consequently, 
$$
\limsup_{s \to 0^+} \frac{\lambda_{k,s}(\Omega)- 1}{s} \le \max\limits_{u\in S_V}\cE_{\g}(u,u).
$$
Since $V$ was chosen arbitrarily, the characterization of the Dirichlet eigenvalues of $(I-\Delta)^{\log}$ given in \eqref{char-eigen}  implies that 
\begin{equation}\label{derivative-k}
\limsup_{s\rightarrow 0^+} \frac{\lambda_{k,s}(\Omega)- 1}{s} \le \inf_{\substack{V\subset C^2_c(\Omega) \\ \dim(V)=k}}\max_{\substack{u\in V\\ \|u\|_{L^2(\Omega)}=1}}\cE_{\g}(u,u) = \lambda_{k,\log}(\Omega).
\end{equation}
This shows that the  inequality in \eqref{eq:lim-inf-limsup-prelim} holds. It   follows directly from \eqref{uniform-upper-bound} that 
\[
\limsup_{s\to 0^+}\lambda_{k,s}(\Omega)\le 1\qquad\text{ for all $k \in \N$}.
\]
From Remark  \ref{Remark}  we have that
$\lambda_{k,s}(\Omega)\ge  \lambda^F_{1,s}(\Omega)$. It  therefore follows from \cite[Lemma 2.8]{RefPST} that 
\[
\liminf_{s\to 0^+}\lambda_{k,s}(\Omega)\ge 1\qquad\text{ for all $k \in \N$}.
\] 
This  proves \eqref{eq:lim-lambda-k-to-1} and  the proof of Lemma \ref{upper-est-lambda_k} is completed.
\end{proof}
\begin{lemma}\label{prop-b}
Let $k \in \N$. If  $\psi_{k,s} \in \cH^s_0(\Omega)$ denote an $L^2$-normalized eigenfunction of $(I-\Delta)^s$, then the set 
$$
\{ \psi_{k,s}\::\: s \in (0,1)\}
$$
is uniformly bounded in $\cH^{\log}_0(\Omega)$ and therefore relatively compact in $L^2(\Omega)$.
 \end{lemma}
 \begin{proof}
To ease notation, we set $\psi_{s}:=\psi_{k,s}$,  the $k$-th  $L^2$-normalized eigenfunction corresponding to $\lambda_{k,s}(\Omega)$, $k\in \N$.  By  \eqref{eq:lim-lambda-k-to-1}, there exits a constant $C=C(N,\Omega,k)>0$  such  that 
\begin{align*}
C\ge \frac{\lambda_{k,s}(\Omega)-1}{s}= \frac{\cE_{\g,s}(\psi_s,\psi_s)-1}{s}&=\int_{\R^N}\frac{(1+|\xi|^2)^s-1}{s}|\psi_s(\xi)|^2\ d\xi\\
&= \int_0^1\int_{\R^N}\log(1+|\xi|^2)|\psi_s(\xi)|^2(1+|\xi|^2)^{st}\ d\xi dt\\
&\ge \frac{1}{2}\int_0^1\int_{\R^N}\log(1+|\xi|^2)|\psi_s(\xi)|^2\ d\xi dt = \frac{1}{2}\cE_{\g}(\psi_s,\psi_s).
\end{align*}
Therefore,  there exist a constant $M:=M(\Omega,k,N)>0$ such that 
\begin{equation}\label{bound-egenf-k}
\sup_{s\in (0,1)}\|\psi_s\|_{\cH^{\log}(\Omega)}\le M
\end{equation}
We  conclude from \eqref{bound-egenf-k} that $\psi_s$ remains uniformly bounded in $\cH_0^{\log}(\Omega)$  for   $s\in(0,1)$.  Consequently $\{\psi_{k,s}: s\in (0,1)\}$ is uniformly bounded in $\cH^{\log}_0(\Omega)$ and relatively compact in $L^2(\Omega)$ since  we have from \eqref{emb} that  $\cH_0^{\log}(\Omega)\hookrightarrow L^2(\Omega)$ is compact.
\end{proof}
We now give the 
\begin{proof}[Proof of Theorem \ref{convergent-1}] The  proof follows the idea  in article  \cite[Theorem 2.10]{RefPST} by the author combined with  \cite[Theorem 3.5]{RefHT}. 
It then suffices, in view of Lemma~\ref{upper-est-lambda_k}, to consider an arbitrary sequence $(s_n)_n\subset(0,1)$ with $\lim\limits_{n\to \infty}s_n=0$, and to show that, after passing to a subsequence, 
\begin{equation}\label{derivative-k2}
\lim_{n \to \infty} \frac{\lambda_{k,s_n}(\Omega)- 1}{s}= \lambda_{k,\log}(\Omega) \qquad \text{for $k \in \N$.}
\end{equation}
Let $\{\psi_{k,s_n}\::\: k \in \N\}$ be an orthonormal system of eigenfunctions corresponding to the Dirichlet eigenvalue $\lambda_{k,s_n}(\Omega)$ of $(I-\Delta)^{s_n}$. By Lemma \ref{prop-b}, it follows that, for every $k \in \N$, the sequence of functions 
$\psi_{k,s_n}$, $n\in \N$ is bounded in $\cH^{\log}_0(\Omega)$ and relatively compact in $L^2(\Omega)$. Consequently, we may pass to a subsequence such that, for every $k \in \N$, 
\begin{equation}
   \psi_{k,s_n} \rightharpoonup \psi^{\star}_{k,\log}\:\text{ weakly  in  $\cH^{\log}_0(\Omega)$} \;\: \text{and}\;\:  \psi_{k,s_n} \rightarrow \psi^{\star}_{k,\log}\: \text{ strongly  in $L^2(\Omega)$}\quad \text{as $n \to \infty$.}\label{conv-of-phi}
\end{equation}
Moreover, by Lemma~\ref{upper-est-lambda_k}, we may, after passing again to a subsequence if necessary, assume that, for every $k \in \N$, 
\begin{equation}
\displaystyle \frac{\lambda_{k,s_n}(\Omega)-1}{s_n}\: \rightarrow \: \lambda^{\star}_k\in \Big[-\infty,\lambda_{k,\log}(\Omega)\Big] \qquad \text{as $n \to \infty$.} \label{conv-of-lamda}
\end{equation}
To prove  then \eqref{derivative-k2}, it now suffices to show that 
\begin{equation}\label{derivative-k2-sufficient}
\lambda_{k,\log}(\Omega)= \lambda^{\star}_k \qquad \text{for every $k \in \N$.}
\end{equation}
It follows from \eqref{conv-of-phi} that
\begin{equation}
  \label{eq:orthonormal-limit-system}
\|\psi^{\star}_{k,\log}\|_{L^2(\Omega)} = 1 \quad \text{and} 
\quad \langle\psi^{\star}_{k,\log},\psi^{\star}_{\ell,\log}\rangle_{L^2(\Omega)}= 0\qquad \text{for $k, \ell \in \N$, $\ell \not = k$.}
\end{equation}
Moreover, for $v\in C^2_c(\Omega)$ and $n \in \N$,  we have from Theorem \ref{eigenvalue} that
\begin{equation}\label{test-phi-k-s}
 \cE_{\g, s_n}(\psi_{k,s_n},v) = \lambda_{k,s_n}(\Omega)\langle\psi_{k,s_n},v\rangle_{L^2(\Omega)}
\end{equation} 
and therefore, rearranging \eqref{test-phi-k-s}, it follows  from  $(i)$ in  Theorem \ref{Result-1} with $p=2$ that
\begin{equation}\label{conv-of-q}
\begin{split}
\lim_{n\rightarrow \infty}\frac{\lambda_{k,s_n}(\Omega)-1}{s_n}\langle \psi_{k,s_n},v\rangle_{L^2(\Omega)}&= \lim_{n\rightarrow \infty}\frac{1}{s_n}\Big( \cE_{\g,s_n}(\psi_{k,s_n},v)-\langle\psi_{k,s_n},v\rangle_{L^2(\Omega)}\Big) \\
&=\lim_{n\rightarrow\infty}\Big\langle\psi_{k,s_n},\frac{(I-\Delta)^{s_n}v-v}{s_n}\Big\rangle_{L^2(\Omega)} \\
&=\langle \psi^{\star}_{k,\log},\logrel v\rangle_{L^2(\Omega)} = \cE_{\g}(\psi^{\star}_{k,\log},v).
\end{split}
\end{equation}
Since moreover $\langle\psi_{k,s_n},v\rangle_{L^2(\Omega)}\to \langle\psi^{\star}_{k,\log},v\rangle_{L^2(\Omega)}$  as $n\to \infty$ for any $k\in \N$ and $v\in C^2_c(\Omega)$, in particular, for $k=1$, we may choose $v\in C^2_c(\Omega)$ such that $\langle\psi^{\star}_{1,\log},v\rangle_{L^2(\Omega)}>0$.  It follows  from \eqref{conv-of-lamda} and  \eqref{conv-of-q} that $\lambda^{\star}_1$  satisfies $ -\infty<\lambda^{\star}_1\le \lambda_{1,\log}(\Omega)$ and
\begin{equation}\label{phi-s0}
\cE_{\g}(\psi^{\star}_{1,\log},v) = \lambda^{\star}_1 \langle \psi^{\star}_{1,\log}, v \rangle_{L^2(\Omega)} ~~~~\text{ for all $v\in \H_0^{\log}(\Omega)$.}
\end{equation}
 Thus $\psi^{\star}_{1,\log}$ is an eigenfunction of $\logrel$ corresponding to the eigenvalue $\lambda^{\star}_1$. Since $\lambda^{\star}_1\le \lambda_{1,\log}(\Omega)$,  it follows  from the definition of the principal eigenvalue  \eqref{Lambda-1} that $\lambda^{\star}_1=\lambda_{1,\log}(\Omega)$ and then  $\lambda_{1,\log}(\Omega)=\lambda^{\star}_1\le \displaystyle\liminf_{s\to 0^+}\frac{\lambda_{1,s}(\Omega)-1}{s}$.  From the uniqueness of the first eigenfunction,  we get that $ \psi^{\star}_{1,\log}=\psi_{1,\log}$ is  the nonnegative  $L^2$-normalized eigenfunction of $\logrel$ corresponding to $\lambda_{1,\log}(\Omega)$. In short, we have just shown that as $s\to 0^+$,
 \[
 \frac{\lambda_{1,s}(\Omega)-1}{s}\to \lambda_{1,\log}(\Omega)\quad\text{ and }\quad \psi_{1,s}\to \psi_{1,\log}\quad \text{ in }\quad L^2(\Omega).
 \]
 This  completes the proof for $k=1$.
Now for $k\ge 2$,    it still follows from \eqref{conv-of-lamda} and   \eqref{conv-of-q}  that 
\begin{equation}\label{phi-s0}
\cE_{\g}(\psi^{\star}_{k,\log},v) = \lambda^{\star}_k \langle \psi^{\star}_{k,\log}, v \rangle_{L^2(\Omega)} ~~~~\text{ for all $v\in C^2_c(\Omega)$,}
\end{equation}
where  $\psi^{\star}_{k,\log}$ is a Dirichlet eigenfunction of  $\logrel$ corresponding to $\lambda^{\star}_k$,  now with 
\begin{equation}\label{L-k}
\lambda^{\star}_k\in [\lambda_{1,\log}(\Omega),\lambda_{k,\log}(\Omega)].
\end{equation}
Next, for fixed $k \in \N$  we consider $E^{\star}_{k}:=\text{span}\{\psi^{\star}_{1,\log},\psi^{\star}_{2,\log},\cdots,\psi^{\star}_{k,\log}\}$, which is a $k$-dimensional subspace of $\cH^{\log}_0(\Omega)$  by \eqref{eq:orthonormal-limit-system}. Since 
\[
\lambda^{\star}_1\le \lambda^{\star}_2\leq \ldots \leq \lambda^{\star}_k
\]
as a consequence of \eqref{L-k} and since $\lambda_{i,s_n}(\Omega)\leq \lambda_{j,s_n}(\Omega)$ for $1\leq i\leq j\leq k$, $n\in\N$, we have the following estimate for every $v =\sum\limits_{i=1}^k\alpha_{i}\psi^{\star}_{i,\log} \in E^{\star}_{k}$ with $\alpha_{1},\cdots,\alpha_{k}\in \R$:
\begin{align}
  \cE_{\g}(v,v) &=  \sum_{i,j=1}^k\alpha_{i}\alpha_{j}\cE_{\g}(\psi^{\star}_{i,\log},\psi^{\star}_{j,\log})=
               \sum_{i,j=1}^k\alpha_{i}\alpha_{j}\lambda_{i}^{\star} \langle \psi^{\star}_{i,\log},\psi^{\star}_{j,\log} \rangle_{L^2(\Omega)}\\
& = \sum_{i=1}^k \alpha_{i}^2 \lambda_{i}^{\star} \|\psi^{\star}_{i,\log}\|_{L^2(\Omega)}^2 \le \lambda_{k}^{\star} \sum_{i=1}^k \alpha_{i}^2 = \lambda_{k}^{\star} \|v\|_{L^2(\Omega)}^2.   
\label{sum}
\end{align}
The characterization in \eqref{char-eigen} now yields that 
\[
\lambda_{k,\log}(\Omega) \le \max_{\substack{v\in E^{\star}_{k} \\ \|v\|_{L^2(\Omega)}=1}}\cE_{\g}(v,v) \le \lambda_{k}^{\star}.
\]
Since also $\lambda_{k}^{\star} \le \lambda_{k,\log}(\Omega)$ by \eqref{conv-of-lamda}, the equality in  \eqref{derivative-k2-sufficient} follows. We thus conclude that \eqref{derivative-k2} holds and also \eqref{asymptotic-exp} follows. Moreover, the statement \eqref{convergent-eigen-k} of the theorem  follows a posteriori from the equality $\lambda_{k}^{\star} = \lambda_{k,\log}(\Omega)$, since we have already seen that $\psi_{k,s_n} \to \psi^{\star}_{k,\log}$ in $L^2(\Omega)$,  the proof is thus finished here.
\end{proof}

\section{ Decay Estimates}\label{Sect-Regularity}
This section  deals with  the proof of     Proposition \ref{Decay} concerning the  decay estimates at infinity and at zero of the  solution $u$ corresponding to Poisson problem,
\begin{equation}\label{Eq:fund-solution-1}
(I-\Delta)^{\log}u =f\qquad\text{in}\quad \R^N.
\end{equation}
The fundamental solution of equation \eqref{Eq:fund-solution-1} can be is given in term of  the Green  function $G: \R^N\setminus\{0\}\to \R$ (see\eqref{Green}) defined by
$$
G(x)= \int_{0}^{\infty}\frac{1}{\Gamma(t)}\int_0^{\infty} p_s(x)s^{t-1}e^{-s}\ ds dt,
$$
We have  in the sense of distributional  that  $\cF(G)(\xi)= \frac{1}{\log(1+|\xi|^2)}$, ~ $\xi \in \R^N\setminus\{0\}$. Indeed, for $\phi\in S$, we have by Fubini's theorem  that
\begin{align*}
\int_{\R^N}G(\xi)\cF(\phi)(\xi)\ d\xi &= \int_{0}^{\infty}\frac{1}{\Gamma(t)}\int_0^{\infty} \int_{\R^N} p_s(\xi)\cF(\phi)(\xi)\ d\xi s^{t-1}e^{-s}\ ds dt\\
&= \int_{\R^N} \int_{0}^{\infty}\frac{1}{\Gamma(t)}\int_0^{\infty} e^{-s(1+|\xi|^2)} s^{t-1}\ ds dt\ \phi(\xi)\ d\xi\\
&=\int_{\R^N} \int_{0}^{\infty}(1+|\xi|^2)^{-t} dt\ \phi(\xi)\ d\xi= \int_{\R^N} \frac{1}{\log(1+|\xi|^2)}\ \phi(\xi)\ d\xi,
\end{align*}
 and then
\[
\cF^{-1}\Big( \frac{1}{\log(1+|\xi|^2)}\Big)(x)= G(x)\quad \text{ for  } x\in \R^N\setminus\{0\}.
\]
We  then define the  solution $u$ of  equation \eqref{Eq:fund-solution-1}  for a $f\in \cC_c^{\infty}(\R^N)$ by
\begin{equation}\label{fund-solution}
u(x)= [G\ast f](x)= \int_{\R^N} G(x-y)f(y) \ dy\quad \text{ for } x\in\R^N.
\end{equation}
This follows from the property of   Fourier transform  and convolution since 
\[
\cF(u) = \cF(G)\cF(f)\quad \text{ and }\quad \log(1+|\xi|^2)\cF(u)= \log(1+|\xi|^2)\cF(G)\cF(f)=\cF(f).
\]
We now  give the 
\begin{proof}[Proof of Proposition \ref{Decay}] 
For $|x|$ small,  We split the integral representation of $G$ in two pieces as follows
\[
G_1(x) =\frac{2^{1-N}}{\pi^{N/2}}\int_{0}^{\frac{N}{2}}\frac{1}{\Gamma(t)}\left(\frac{|x|}{2}\right)^{t-\frac{N}{2}}K_{t-\frac{N}{2}}(|x|)\ dt 
\]
and
\[
 G_2(x)=\frac{2^{1-N}}{\pi^{N/2}}\int_{\frac{N}{2}}^{\infty}\frac{1}{\Gamma(t)}\left(\frac{|x|}{2}\right)^{t-\frac{N}{2}}K_{t-\frac{N}{2}}(|x|)\ dt.
\]
Since $t\le \frac{N}{2}$, it  follows from  the asymptotics property \eqref{M-Bessl-Asymtotic} for $K_{\nu}$   (see \cite{RefRE}) that as  ~$|x|\to 0$,
\[
K_{t-\frac{N}{2}}(|x|)\sim 2^{|t-\frac{N}{2}|-1}\Gamma(|t-\frac{N}{2}|)|x|^{-|t-\frac{N}{2}|} \sim \begin{cases}2^{\frac{N}{2}-t-1}\Gamma(N/2-t)|x|^{-\frac{N}{2}+t}\quad\text{if } \quad t<\frac{N}{2} ,\\
\log\frac{1}{|x|} \quad\text{if } \qquad t=\frac{N}{2}.
\end{cases}
\]
Plugging the above approximations in  $G_1$,  we  end up with
 \begin{equation}\label{G_1}
	\begin{split}
	G_1(x)\sim\quad\left\{\begin{aligned}
		\frac{2^{1-N}}{\pi^{N/2}}\log\frac{1}{|x|}\hspace{2.5cm}& \quad\text{  as }\quad |x|\to 0\quad\text{ if }\quad t=\frac{N}{2}\\
    \frac{2^N}{\pi^{N/2}}|x|^{-N} \int_{0}^{\frac{N}{2}}\frac{\Gamma(N/2-t)}{4^t\Gamma(t)}\ dt \hspace{0.cm}& \quad\text{  as }\quad |x|\to 0\quad\text{ if }\quad t< \frac{N}{2},
	\end{aligned}\right.
	\end{split}
	\end{equation}
	where we have used that since $N>2t$, $|x|^{-N+2t}\sim |x|^{-N}$ as $|x|\to 0$. Since also $ t< \frac{N}{2}$, we have $\int_{0}^{\frac{N}{2}}\frac{\Gamma(N/2-t)}{4^t\Gamma(t)}\ dt<\infty$.
	Now, for $t>\frac{N}{2}$,  again by using \eqref{M-Bessl-Asymtotic},  we have 
\begin{align*}
K_{t-\frac{N}{2}}(|x|)&\sim 2^{t-\frac{N}{2}-1}\Gamma(t-\frac{N}{2})|x|^{-t+\frac{N}{2}} \quad\text{ as }\quad|x|\to 0.
\end{align*}
Taken the above approximations into account, we get the approximation for $G_2$, 
\begin{equation}\label{G_2}
G_2(x)\sim \frac{2^{-N}}{\pi^{N/2}}\int_{\frac{N}{2}}^{\infty}\frac{\Gamma(t-N/2)}{\Gamma(t)}\ dt\quad\text{ as }\qquad|x|\to 0
\end{equation}
Since  $\displaystyle\lim_{t\to +\infty} \frac{\Gamma(t-\frac{N}{2})}{4^t\Gamma(t)}=0$ and $t>\frac{N}{2}$, we infer that $\int_{2N}^{\infty}\frac{\Gamma(t-\frac{N}{2})}{\Gamma(t)}\ dt<\infty$. Therefore,  
  combining the approximations of $G_1$ in \eqref{G_1} and $G_2$ in \eqref{G_2} we get 
\[
|x|^NG(x) \sim\frac{2^N}{\pi^{N/2}}\int_{0}^{\frac{N}{2}+1}\frac{\Gamma(N/2-t)}{4^t\Gamma(t)}\ dt  \quad \text{ as } \quad |x|\to 0. 
\]
We next investigate the case with the modulus of  $|x|$ large. From the asymptotics property \eqref{M-Bessl-Asymtotic}  we have for all $t\ge 0$ that
\begin{align*}
|x|^{t-\frac{N}{2}}K_{t-\frac{N}{2}}(|x|)&\sim \frac{\pi^{\frac{1}{2}}}{\sqrt{2}}|x|^{-\frac{N+1}{2}+t}e^{-|x|}\quad \text{ as } \quad |x|\to \infty\\
&\sim \frac{\pi^{\frac{1}{2}}}{\sqrt{2}}|x|^{-\frac{N+1}{2}}e^{-|x|}\qquad \text{ as } \quad |x|\to \infty. 
\end{align*}
From this, we infer that
\begin{align*}
G(x)& \sim        ~2^{-\frac{N+1}{2}}\pi^{-\frac{N-1}{2}}|x|^{-\frac{N+1}{2}}e^{-|x|}\int_{0}^{\infty}\frac{1}{2^{t}\Gamma(t)}\ dt \quad \text{ as }\quad |x|\to \infty.
\end{align*}
Noticing that $\displaystyle\lim_{t\to 0} 1/(2^t\Gamma(t))=0=\lim_{t\to +\infty} 1/(2^t\Gamma(t))$, the above integral is finite and 
\[
\int_{0}^{\infty}\frac{1}{2^{t}\Gamma(t)}\ dt\sim 1.
\]
We therefore infer that
\[
G(x)\sim ~2^{-\frac{N+1}{2}}\pi^{-\frac{N-1}{2}}|x|^{-\frac{N+1}{2}}e^{-|x|}\quad \text{ as } |x|\to \infty.
\]
For $f\in L^1(\R^N)$, we write 
\begin{align*}
u(x)=  \int_{\R^N}G(x-y) f(y)\ dy=\int_{\R^N}G(y) f(x-y)\ dy 
\end{align*}
First observe that if $f\ge 0$, we have that 
\begin{align*}
u(x)\ge \int_{B(x,|x|)} G(x-y) f(y)\ dy \ge Ce^{-|x|}\int_{B(x,|x|)}  f(y)\ dy.
\end{align*}
Since $B(x,|x|)\to \R^N$ as $|x|\to \infty$ and $f\in L^1(\R^N)$, we see that 
$
u(x)=O(e^{-|x|})
$
as $|x|\to \infty$. 
Moreover, Since $G(x)$ decays  as $e^{-|x|}$ at infinity, there exists a constant $M>0$ such that 
\[
\| e^{|\cdot|}G\|_{L^{\infty}(\R^N)}<C \quad\text{ for }\quad |x|\ge M,
\]
where $C>0$ is a positive constant. We then write
\begin{align*}
e^{|x|}u(x)=  [e^{|\cdot|}G\ast f](x)&=\int_{\R^N}e^{|y|}G(y) f(x-y)\ dy.
\end{align*}
 Thus,
\begin{align*}
|e^{|x|}u(x)|\le \Big| \int_{\R^N}e^{|y|}G(y) f(x-y)\ dy\Big|
& \le \|e^{|\cdot|}G\|_{ L^{\infty}(\R^N)}\int_{\R^N} |f(x-y)|\ dy\\
&\le C\|f\|_{L^1(\R^N)}.
\end{align*}
This allows to conclude that $u(x)$ decays  as $e^{-|x|}$ at infinity, that is 
\[
u(x)= O(e^{-|x|}) \qquad\text{ as}\quad |x|\to \infty.
\]
As before, there exists $\delta>0$ such that   
\[
\||\cdot|^{N}u\|_{L^{\infty}(\R^N)}<C \quad \text{ for } |x|<\delta.
\]
Therefore,
\begin{align*}
||x|^Nu(x)|&\le  C\int_{\R^N} |f(x-y)|\ dy \le C\|f\|_{L^{1}(\R^N)}.
\end{align*}
This allows to conclude that
$$
u(x)=O(|x|^{-N}) \quad\text{ as } \quad |x|\to 0.
$$
This completes the proof of Theorem \ref{Decay}.
\end{proof}

\section{Additional remarks}\label{Additional}
We present in this section some results concerning the logarithmic Schr\"odinger operator  $\logrel$ that can be directly deduced  from known results in the literature.
For this fact, we introduce the following space $\cV_{\g}(\Omega)$, being the space of all functions $u\in L^2_{loc}(\R^N)$ such that 
\[
\rho(u,\Omega) :=\int_{\Omega}\int_{\R^N}\frac{(u(x)-u(y))^2}{|x-y|^{N}}\omega(|x-y|)\ dxdy<\infty.
\]
Then the quantity $\cE_{\g}(u,v)$ is well dined for $u\in \cH^{\log}_0(\Omega)$ and $v\in \cV_{\g}(\Omega)$ (see \cite[Lemma 3.1]{RefJ2}). The proof of the following results on the maximum principle for the operator $(I-\Delta)^{\log}$ on an open set $\Omega$ of $\R^N$ can be  deduced from \cite{RefJ2}.
\begin{thm}
\begin{itemize}
\item[(i)](Strong maximum principle)
Let $\Omega\subset \R^N$ be a bounded subset and $u\in \cL_0(\R^N)$ be a continuous function on $\overline{\Omega}$ satisfying 
\[
(I-\Delta)^{\log}u\ge 0\quad\text{ in} \quad\Omega, \qquad u\ge 0 \quad \text{ in }\quad\R^N\setminus\Omega.
\]
Then \ $u>0$ \ in \ $\Omega$ \ or \ $u\equiv 0$ \ a.e. \ in \ $\R^N$.
\item[(ii)] (Weak maximum principle)
Let $u\in \cV_{\g}(\Omega)$ with $(I-\Delta)^{\log}u\ge 0$ in $\Omega$ weakly and $u\ge 0$ in $\R^N\setminus \Omega$. Then $u\ge 0$ in $\R^N$.
\item[(iii)](Small volume maximum principle)There exists $\delta> 0$ such that for every open bounded set $\Omega$ of $\R^N$ with $|\Omega|\le \delta$ and every function $u\in \cV_{\g}(\Omega)$ satisfying 
\[
(I-\Delta)^{\log}u\ge c(x)u \quad\text{ in }\  \Omega \quad\text{ and } \quad u\ge 0 \ \text{ in }\quad\R^N\setminus \Omega,
\]
with $c\in L^{\infty}(\R^N)$, then $u\ge 0$ in $\R^N$.
\end{itemize}
\end{thm}
We recall that,  $u\in\cV_{\g}(\Omega)$ satisfies  $(I-\Delta)^{\log}u\ge 0$ in $\Omega$ weakly means, 
$$\text{$\cE_{\g}(u,\phi)\ge 0$ for all nonnegative $\phi\in \cC^{\infty}_c(\Omega)$. }$$
Next, consider the following semilinear elliptic problem involving the operator $(I-\Delta)^{\log}$ in a bounded set $\Omega$ of $\R^N$,
\begin{equation}\label{eq-nonlinear}
	\begin{split}
		(I-\Delta)^{\log}u  =   f(x,u)  \text{ \ \   in \ \  $\Omega$}\qquad
		u =0              \text{ on }  \mathbb{R}^N\setminus \Omega,
	\end{split}
	\end{equation}
where  $f:\Omega\times\R\to\R$ is continuous. 
The following result on the radially symmetry of the solution can deduced from \cite{RefST1}

\begin{thm}\label{Nonlinear-result}
 Assume that $f$ is locally Lipschitz with respect to the second variable and  radially symmetry and strictly decreasing in $r=|x|$. Then every positive solution of \eqref{eq-nonlinear} is radially symmetry and strictly decreasing in $|x|$.
\end{thm}

\bibliographystyle{amsplain}

\end{document}